\documentclass[reqno]{amsart}
\title{The geometry of some generalized affine Springer fibers}
\author{Jingren Chi}
\address{Department of Mathematics, The University of Chicago,
5734 S. University Ave., Chicago, IL, 60637}
\email{chijingren@gmail.com}
\usepackage{amsmath,amsthm,amssymb,amscd}
\usepackage[all,cmtip]{xy}
\usepackage{dsfont}
\usepackage{mathrsfs}
\usepackage{leftidx}
\usepackage[colorinlistoftodos]{todonotes}
\usepackage[square,sort,comma,numbers]{natbib}

\usepackage{color}
\usepackage{hyperref}
\hypersetup{
   colorlinks=true,
   linktoc=all,
   linkcolor=black,
   citecolor=black,
   urlcolor=black,
}

\newcommand{\into}{\hookrightarrow}

\newcommand{\fg}{\mathfrak{g}}

\newcommand{\bA}{\mathbb{A}}

\newcommand{\bB}{\mathbb{B}}

\newcommand{\fc}{\mathfrak{c}}
\newcommand{\fC}{\mathfrak{C}}

\newcommand{\fD}{\mathfrak{D}}
\newcommand{\sfD}{\mathsf{D}}

\newcommand{\ga}{\gamma}
\newcommand{\bG}{\mathbb{G}}

\newcommand{\cH}{\mathcal{H}}

\newcommand{\cI}{\mathcal{I}}

\newcommand{\cJ}{\mathcal{J}}

\newcommand{\la}{\lambda}
\newcommand{\La}{\Lambda}

\newcommand{\bN}{\mathbb{N}}
\newcommand{\cN}{\mathcal{N}}

\newcommand{\cO}{\mathcal{O}}

\newcommand{\cP}{\mathcal{P}}
\newcommand{\sfP}{\mathsf{P}}

\newcommand{\bQ}{\mathbb{Q}}

\newcommand{\ft}{\mathfrak{t}}

\newcommand{\bZ}{\mathbb{Z}}

\def\spec{\mathop{\rm Spec}\nolimits}
\def\Hom{\mathop{\rm Hom}\nolimits}

\newtheorem{thm}[subsubsection]{Theorem}
\newtheorem*{thm*}{Theorem}
\newtheorem{cor}[subsubsection]{Corollary}
\newtheorem{lem}[subsubsection]{Lemma}
\newtheorem{prop}[subsubsection]{Proposition}

\theoremstyle{definition}
\newtheorem{defn}[subsubsection]{Definition}

\newtheorem{conjecture}[subsubsection]{Conjecture}
\newtheorem*{conjecture*}{Conjecture}
\theoremstyle{remark}
\newtheorem{rem}[subsubsection]{Remark}

\numberwithin{equation}{section}
\newcommand{\introthmname}{}
\newtheorem{introthminn}{\introthmname}

\begin{document}
\begin{abstract}
We study basic geometric properties of some group analogue of affine Springer fibers and compare with the classical Lie algebra affine Springer fibers. The main purpose is to formulate a conjecture that relates the number of irreducible components of such varieties for a reductive group $G$ to certain weight multiplicities defined by the Langlands dual group $\hat{G}$. We prove our conjecture in the case of unramified conjugacy class.
\end{abstract}

\maketitle
\section*{Introduction}
\subsection{Background and motivation}
The ``generalized affine Springer fibers" in the title refers to sets of the following form
\[X_\gamma^\la=\{g\in G(F)/G(\cO)| g^{-1}\gamma g\in G(\cO)\varpi^\la G(\cO)\}\]
where 
\begin{itemize}
\item $G$ is a split connected reductive algebraic group over a field $k$;
\item $F=k((\varpi))$ is the field of Laurent series with coefficients in $k$ and $\cO=k[[\varpi]]$ is the ring of power series;
\item $\gamma\in G(F)$ is a regular semisimple element;
\item $\la:\bG_m\to T$ is a cocharacter of a maximal torus $T$ of $G$ and 
\[\varpi^\la:=\la(\varpi)\in G(F).\]
\end{itemize}
When $k$ is a finite field, the set $X_\ga^\la$ arises naturally in the study of orbital integrals of functions in the spherical Hecke algebra $\cH(G(F), G(\cO))$, which consists of $G(\cO)$-biinvariant locally constant functions with compact support on $G(F)$.\par 
It turns out that $X_\gamma^\la$ can be realized as the set of $k$-rational points of some algebraic variety over $k$. This kind of variety has previously been studied by Kottwitz-Viehmann in \citep{KoV} and Lusztig in \citep{Lu15}. The adjective ``generalized" refers to the fact that  the varieties $X_\ga^\la$ could be viewed as group analogue of some affine Springer fibers for Lie algebras studied by Kazhdan and Lusztig in \cite{KL}:
\[X_\ga=\{g\in G(F)/G(\cO)|\mathrm{ad}(g)^{-1}\gamma\in\mathfrak{g}(\cO)\}.\]
Here $\fg$ is the Lie algebra of $G$, $\ga\in\fg(F)$ is a regular semisimple element and ``ad" denotes the adjoint action of $G$ on $\fg$.\par 
Basic geometric properties of these Lie algebra affine Springer fiber $X_\ga$ (dimension, irreducible components etc.) have been well understood through the works of Kazhdan and Lusztig \cite{KL}, Bezrukavnikov \cite{Be96}, Ng\^o \cite{Ngo10}. A key ingredient in their approach is the symmetry on $X_\gamma$ arising from the centralizer $G_\ga(F)$. More precisely, $X_\ga$ has an action of a commutative algebraic group $P_\gamma$ locally of finite type over $k$, which is defined as a quotient of the loop group $G_\ga(F)$. There is an open dense subset $X_\ga^{\mathrm{reg}}$ of $X_\ga$ (the ``regular locus") on which $P_\ga$ acts simply transitively. Hence the dimension of $X_\ga$ equals the dimension of $P_\ga$, the latter of which has been calculated by Bezrukavnikov in \cite{Be96}. Moreover, $X_\ga$ is equidimensional and its set of irreducible components is in bijection with $\pi_0(P_\ga)$, the set of connected components of $P_\ga$. \par 
We would like to generalize the above picture for Lie algebras to the group analogue $X_\ga^\la$. In this case, one can still construct an action of a commutative algebraic group $P_\ga$ and define an open subset $X_\ga^{\la,\mathrm{reg}}$ (the ``regular locus") having the same dimension as $P_\ga$.  However, there are the following notable differences from the Lie algebra situation:
\begin{itemize}
\item In general the action of $P_\ga$ on $X_\ga^{\la,\mathrm{reg}}$ is not transitive.
\item A more serious problem is that in general the ``regular locus" $X_\ga^{\la,\mathrm{reg}}$ is not dense in $X_\ga$ and there might be irreducible components disjoint from $X_\ga^{\la,\mathrm{reg}}$.
\end{itemize}
Thus $X_\ga^\la$ may have more irreducible components than $P_\ga$ and to calculate its dimension, it is not sufficient to calculate the dimension of $P_\ga$. 
\subsection{Main results}
The first goal of this paper is to establish some basic geometric properties of $X_\ga^\la$. We prove a dimension formula of $X_\ga^\la$ when $\ga$ is unramified. For general $\ga$, we establish the dimension formula for the regular open subset $X_\ga^{\la,\mathrm{reg}}$. We leave the proof of the dimension formula in full generality to \citep{BC17}. A previous attempt on dimension formula has been made in \cite{Bou15}, but there are some gaps in \textit{loc. cit.}, see the discussion in \citep{BC17}.

\begin{thm*}
$X_\ga^\la$ is a finite dimensional $k$-scheme locally of finite type. There is an equidimensional open subscheme $X_\ga^{\la,\mathrm{reg}}\subset X_\ga^\la$ whose dimension is given by
\[\dim X_\ga^{\la,\mathrm{reg}}=\langle\rho,\la\rangle+\frac{1}{2}(d(\ga)-c(\ga))\]
where 
\begin{itemize}
\item $\rho$ is half sum of the positive roots for $G$;
\item $d(\ga)=\mathrm{val}(\det(\mathrm{Id}-\mathrm{ad}_\ga:\fg(F)/\fg_{\ga}(F)\to\fg(F)/\fg_{\ga}(F)))$ is the discriminant valuation of $\ga$;
\item $c(\ga)=\mathrm{rank}(G)-\mathrm{rank}_F(G_\ga)$ is the difference between the dimension of the maximal torus of $G$ and the dimension of the maximal $F$-split subtorus of the centralizer $G_\ga$.
\end{itemize}
\end{thm*}

This is proved in Theorem~\ref{finite-type-thm} and Theorem~\ref{dim-reg-locus-thm}.\par 
Regarding the dimension of $X_\ga^\la$ itself, we mention the following

\begin{thm*}[joint with A. Bouthier \citep{BC17}]
$X_\ga^\la$ is equidimensional of dimension
\[\dim X_\ga^\la=\langle\rho,\la\rangle+\frac{1}{2}(d(\ga)-c(\ga))\]
\end{thm*}

By the previous result, it remains to show that
\[\dim X_\ga^{\la,\mathrm{reg}}=\dim X_\ga^\la\]
which is done in \citep{BC17}. We remark that the argument of Kazhdan-Lusztig in \citep{KL} does not generalize to our situation since otherwise it would imply that the complement of the regular open subset has strictly smaller dimension (see \citep[Proposition 3.7.1]{Ngo10}), which in our situation may not be true due to the possible existence of irregular components. In general, actually most components of $X_\ga^\la$ will be irregular, see Remark~\ref{regular-components-rem}.\par 
This natually leads to the question of determining the number of irreducible components of $X_\ga^\la$, which is our second goal. We will formulate a conjecture on the number of irreducible components of $X_\ga^\la$ and prove the conjecture in the case where $\ga$ is unramified (or split) conjugacy class. One formulation of the conjecture involves the Newton point $\nu_\ga\in (X_*(T)\otimes\bQ)^+$ of $\ga$, which is an element in the dominant rational coweight cone. For the precise definition, see \S~\ref{Newton-point-section}. We show in \ref{nonempty-prop} that $X_\ga^\la$ is nonempty if and only if $\nu_\ga\le_\bQ\la$ in the sense that $\la-\nu_\ga$ is a $\bQ$ linear combination of positive coroots with non-negative coefficients. Then by the discussion in \S\ref{irr-components-conjecture-section} there exists a unique \emph{smallest} dominant \emph{integral} coweight $\mu$ such that $\nu_\ga\le_\bQ\mu$ and $\mu\le\la$.
\begin{conjecture*}[Conjecture~\ref{irr-components-conjecture}]
Let $\mu$ be as above. The number of $G_\ga(F)$-orbits on the set of irreducible components of $X_\ga^\la$ equals to $m_{\la\mu}$, which is the dimension of $\mu$-weight space in the irreducible representation $V_\lambda$ of the Langlands dual group $\hat G$ with highest weight $\lambda$.  
\end{conjecture*}
We remark that there is a similar conjecture made by Miaofen Chen and Xinwen Zhu on the irreducible components of affine Deligne-Lusztig varieties, see \citep{HaVi17} and \citep{XiaoZhu17} for statements. Besides the similarities, there is a subtle difference between our formulation and the one of Chen-Zhu, see
Remark~\ref{best-integral-approx-rem}.\par 
In fact we will also give a better formulation of this Conjecture using the Steinberg quotient of $G$ (``space of characteristic polynomials"). See Conjecture~\ref{irr-components-conjecture} for more details.\par 
\begin{thm*}
The Conjecture is true if either $\la=0$ or $\ga\in G(F)^{\mathrm{rs}}$ is unramified (i.e. split).
\end{thm*}
\begin{proof}
If $\la=0$, then one can adapt the argument of Kazhdan-Lusztig to show that there is only one $G_\ga(F)$-orbit on $\mathrm{Irr}(X_\ga^{\la=0})$. See \citep[\S4.2]{Bou15}. The case where $\ga$ is unramifed is proved in Corollary~\ref{dim-unr-cor}.
\end{proof}
\begin{rem}
Although we restrict to equal characteristic local field, the results in this paper also generalize to mixed characteristic generalized affine Springer fibers, which could be defined after the work of Xinwen Zhu \citep{Zhu17}. However, the dimension formula in full generality still remains open in mixed characteristic case since in \citep{BC17}, we used the group version of Hitchin fibration which does not have a mixed characteristic analogue yet.
\end{rem}
\subsection{Organization of the article}
In \S\ref{monoid-section}, we review certain facts needed from the theory of reductive monoids. We follow the exposition of \cite{Bou15} with some modifications. In \S\ref{fiber-section}, we define the ind-scheme structure on $X_\ga^\la$ and provide criterions for its nonemptiness. Also we study the natural symmetries on $X_\ga^\la$. In \S\ref{unr-section} we prove dimension formula and the conjecture on irreducible components in the unramified case. In \S\ref{general-case-section}, we prove that $X_\ga^\la$ is a finite dimensional scheme locally of finite type and formulate precise conjectures on its irreducible components.

\subsection{Notations and conventions}\label{notation-subsection}
\subsubsection{Group theoretic notations}
Assume throughout the paper that $k$ is an algebraically closed field. $F=k((\varpi))$ and $\cO=k[[\varpi]]$. We let $G$ be a (split) connected reductive group over $k$. Assume that either $\mathrm{char}(k)=0$ or $\mathrm{char}(k)>0$ does not divide the order of Weyl group of $G$.\par 
 Let $G_0$ be the derived group of $G$, a semisimple group of rank $r$. \textbf{We assume that $G_0$ is simply connected throughout the paper.} Let $Z$ be the connected center of $G$ and $Z_0$ the center of $G_0$. In particular, $Z_0$ is a finite abelian group.\par 
Fix a maximal torus $T$ of $G$ and a Borel subgroup $B$ containing $G$. Let $\Delta=\{\alpha_1,\dotsc,\alpha_r\}$ be the set of simple roots determined by $B$. Let $\check{\Lambda}:=X^*(T)$ (resp. $\Lambda:=X_*(T)$) be the weight (resp. coweight) lattice. Let $\check{\Lambda}^+$ (resp. $\Lambda^+$) be the set of dominant weights (resp. dominant coweights). Also let $T_0=T\cap G_0$, $B_0=B\cap G_0$ and $Z_0$ be the center of $G_0$. Moreover, $\Delta$ is also viewed as a set of simple roots for $G_0$.\par 
We have the canonical abelian quotient
\[\det:G\to G_{ab}:=G/G_0=T/T_0=Z/Z_0\]
Let $W$ be the Weyl group of $G$ and $S\subset W$ the set of simple reflections associated to the simple roots $\Delta$. There is a unique longest element $w_0$ of $W$ under the Bruhat order determined by $S$. Then $w_0$ is a reflection and $-w_0$ defines a bijection on the sets $\Delta$, $\Lambda^+$ and $\check\Lambda^+$.\par 
Let $\hat G$ be the Langlands dual group of $G$, viewed as a complex reductive group. For each $\lambda\in\Lambda^+$, viewed as a dominant weight for $\hat G$, let $V(\lambda)$ be the irreducible representation of $\hat G$ with highest weight $\la$. For any $\mu\in\La^+$ with $\mu\le\la$, let $m_{\la\mu}$ be the dimension of $\mu$ weight space in $V(\la)$.\par 
\subsubsection{Scheme theoretic notations}
For any scheme $X$ over $\cO$, we let $L_n^+X$ be its $n$-th jet space. In other words, $L_n^+X$ is the $k$-scheme such that for any $k$ algebra $R$, $L_n^+X(R)$ is the set of morphisms $\spec R[\varpi]/\varpi^n\to X$ whose composition with the structure morphism $X\to\spec\cO$ corresponds to the canonical $k$-algebra homomorphism $\cO=k[[\varpi]]\to R[\varpi]/\varpi^n$.\par 
Let $L^+X:=\varprojlim L^+_nX$ be the arc space and $LX$ its loop space. More precisely, $LX$ is the $k$-functor such that for any $k$-algebra $R$, $LX(R)$ is the set of morphisms $\spec R((\varpi))\to X$ whose composition with the structure morphism $X\to\spec\cO$ corresponds to the canonical $k$-algebra homomorphism $\cO=k[[\varpi]]\to R((\varpi))$.\par
If $X$ is a $k$-scheme, we define $L_n^+X$, $L^+X$, $LX$ using the $\cO$-scheme $X\otimes_k\cO$. \par 
For any scheme $X$, we let $\mathrm{Irr}(X)$ be the set of its irreducible components.

\subsection*{Acknowledgement}
I am deeply grateful to my doctoral advisor Ng\^o Bao Ch\^au for his encouragement and interst in this paper, for sharing many insights and teaching me many things over the years. I thank Alexis Bouthier, Cheng-Chiang Tsai, Zhiwei Yun and Xinwen Zhu for their interest as well as helpful discussions and suggestions. I thank George Lusztig for bringing my attention to \citep{Lu15}. I thank Jim Humphreys, David Speyer, Hugh Thomas for answering my question on MathOverflow related to Corollary~\ref{lower-bound-multiplicity-cor}. Also, I thank Jim Humphreys for valuable feedbacks on the first draft of this paper and helping me clarify the notion of Coxeter elements.

\section{Recollection on reductive monoids}\label{monoid-section}
We keep the notations and assumptions in \S\ref{notation-subsection}. In particular, $G$ is a connected (split) reductive group over $k$ whose derived group $G_0$ is simply connected.\par 
\subsection{Construction of Vinberg monoid}
Let $\omega^0_1,\dotsc,\omega^0_r\in X^*(T_0)_+$ be the fundamental weights dual to the simple coroots. For each $1\le i\le r$, let $\rho_i:G_0\to\mathrm{GL}(V_{\omega^0_i})$ be the irreducible representation with highest weight $\omega^0_i$.
\subsubsection{}
We introduce the enhanced group $G_0^+:=(T_0\times G_0)/Z_0$ where $Z_0$, the center of $G_0$, embeds anti-diagonally in $T_0\times G_0$. Then $G_0^+$ is a reductive group with center $Z_0^+=(T_0\times Z_0)/Z_0\cong T_0$ and derived group $G_0$.  Let $T_0^+=(T_0\times T_0)/Z_0$ be a maximal torus of $G_0^+$ and $B_0^+=(T_0\times B_0)/Z_0$ a Borel subgroup containing $T_0^+$. The representation $\rho_i$ can be extended to a representation of $G_0^+$: 
\[\xymatrix@R=1pt{
\rho_i^+:G_0^+\ar[r] & \mathrm{GL}(V_{\omega^0_i})\\
(t,g)\ar@{|->}[r] & \omega^0_i(t)\rho_i(g)
}\]
For each $1\le i\le r$, we also extend the simple roots $\alpha_i$ to $\alpha_i^+:G_0^+\to\bG_m$ by $\alpha_i^+(t,g)=\alpha_i(t)$. Altogether, we get the following homomorphism
\[(\alpha^+,\rho^+):G_0^+\to\bG_m^r\times\prod_{i=1}^r\mathrm{GL}(V_{\omega^0_i}) \]
which is easily seen to be a closed embedding. 
\begin{defn}
The \emph{Vinberg monoid} of $G_0$, denoted by $V_{G_0}$, is the normalization of the closure of $G_0^+$ in the product 
\[\bA^r\times\prod_{i=1}^r\mathrm{End}(V_{\omega^0_i})\]
\end{defn}
Then $V_{G_0}$ is an algebraic monoid with unit group $G_0^+$. It has a smooth dense open subvariety $V_{G_0}^0$ defined as the normalization of the closure of $G_0^+$ in the product
\[\bA^r\times\prod_{i=1}^r(\mathrm{End}(V_{\omega^0_i})-\{0\})\]
The action of $G_0^+\times G_0^+$ on $G_0^+$ by left and right multiplication extends to an action on $V_{G_0}$. 
\begin{defn}
The \emph{abelianization} of $V_{G_0}$ is the GIT quotient
\[A_{G_0}:=V_{G_0}//(G_0\times G_0)\]
\end{defn}
Let $\alpha:V_{G_0}\to A_{G_0}$ be the quotient map. Then $\alpha$ is smooth and there exists a canonical section of $\alpha$:
\begin{equation}\label{section-alpha-eq}
s: A_{G_0}\to V_{G_0}
\end{equation}
whose image is the closure of the diagonal torus
\[T_0^\Delta:=\{(t,t^{-1})\in T_0\times T_0\}/Z_0\]
in $V_{G_0}$. In fact, $\alpha$ and $s$ induces canonical isomorphisms between $T_0^\Delta$ and $T_0^{\mathrm{ad}}$.
\begin{defn}
The \emph{asymptotic semigroup} of $G_0$ is defined by $\mathrm{As}(G_0):=\alpha^{-1}(0)$.
\end{defn}
\subsubsection{Relating $G$ and $G_0^+$}
It is proved by Vinberg in characteristic 0, and extended to positive characteristic by Rittatore, that any reductive monoid $M$ whose unit group has derived group isomorphic to $G_0$ is pulled back from $V_{G_0}$ by a homomorphism 
\[A_M\to A_{G_0}\] where $A_M$ is the abelianization of $M$, see \cite[Theorem 9]{Rit01}. Applying this result to the case where $M=G$ and $A=G_{\mathrm{ab}}$, we obtain a map $G_{\mathrm{ab}}\to A_{G_0}$ which factors through the unit group $T_0^{\mathrm{ad}}$ of $A_{G_0}$. Thus we obtain the following Catesian diagram
\begin{equation}\label{G-ab-Cartesian-square-eq}
\xymatrix{
G\ar[r]^{\theta}\ar[d]_\det & G_0^+\ar[d]^{\alpha|_{G_0^+}}\\
G_{\mathrm{ab}}\ar[r]^{\theta_{\mathrm{ab}}} & T_0^{\mathrm{ad}}
}\end{equation}

Recall that $G=(Z\times G_0)/Z\cap G_0$. The map $\theta$ induces the identity map on $G_0$ and a canonical homomorphism $\theta_Z:Z\to T_0$ which restricts to identity on the finite group $Z\cap G_0$. Note that in general the map $\theta_{\mathrm{ab}}$ (and hence $\theta$) is not unique. We fix one such map from now on.\par 
The map $\theta$ induces group homomorphisms
\[\theta_*: \La:=X_*(T)\to X_*(T_0^+),\quad\quad\theta^*: X^*(T_0^+)\to \check{\La}:=X^*(T)\]
\subsubsection{}\label{omega-i-section}
The natural quotient homomorphism $T_0\to T_0^{\mathrm{ad}}$ allows us to view $X_*(T_0)$ as a finite index subgroup of $X_*(T_0^{\mathrm{ad}})$. Then we have a concrete description of the coweight lattice of $T_0^+=(T_0\times T_0)/Z_0$ as follows:
\begin{equation}\label{T-plus-coweight-lattice-eq}
X_*(T_0^+)=\{(\nu_1,\nu_2)\in X_*(T_0^{\mathrm{ad}})\times X_*(T_0^{\mathrm{ad}}) | \nu_1+\nu_2\in X_*(T_0)\}.
\end{equation}
Using this description, for each $1\le i\le r$ we define $\omega_i^+:=(\omega_i^0,\omega_i^0)\in X_*(T_0^+)$ and $\omega_i:=\theta^*(\omega_i^+)$. Then $\omega^+_i$ is the highest weight of the $G_0^+$-representation $\rho_i^+$ and hence $\omega_i$ is the highest weight of the $G$-representation $\rho_i^+\circ\theta$.

\subsection{The adjoint quotient}
\subsubsection{}
Recall that the \emph{Steinberg base} is
the GIT quotient
\[\fC_G:=G//\mathrm{Ad}(G)=\spec k[G]^{\mathrm{Ad}(G)}\]
where $\mathrm{Ad}(G)$ denotes the conjugation (or adjoint) action of $G$ on itself. Also we have the Steinberg base $\fC_{G_0}:=G_0//\mathrm{Ad}(G_0)$ for the derived group $G_0$. Denote by $\chi_G: G\to\fC$ and $\chi_{G_0}:G_0\to\fC_{G_0}$ the canonical quotient maps.\par 
Under the assumption that $G_0$ is simply connected, Steinberg shows in \cite{St65} that the conjugation invariant functions $\mathrm{Tr}\rho_i$ define an isomorphism $\fC_{G_0}\xrightarrow{\sim}\bA^r$.\par 
We can extend Steinberg's result to $G$ as follows. For each $1\le i\le r$, the composition $\rho_i^+\circ\theta$, where $\theta$ is as in \eqref{G-ab-Cartesian-square-eq}, defines an irreducible representation of $G$ extending the fundamental representation $V_{\omega^0_i}$ of $G_0$. Its highest weight is $\omega_i:=\theta^*(\omega_i^+)\in\check{\La}$, cf. \S\ref{omega-i-section}. \par 
The maps $\mathrm{Tr}(\rho_i^+\circ\theta)$ and $\det$ define an isomorphism
\begin{equation}\label{Steinberg-base-eq}
\fC_G\xrightarrow{\sim}\bA^r\times G_{\mathrm{ab}}
\end{equation}
which restricts to Steinberg's isomorphism $\fC_{G_0}\simeq\bA^r$.\par 
Moreover, there is a canonical isomorphism $k[G]^{\mathrm{Ad}(G)}\cong k[T]^W$ defined by restriction of functions, hence an isomorphism $\fC_G\cong T//W$.

\subsubsection{}
The \emph{extended Steinberg base for $G_0$}  is defined to be the GIT quotient 
\[\fC_+:=V_{G_0}//G_0\]
where $G_0$ acts by conjugation. Let $\chi_+: V_{G_0}\to\fC_+$ be the canonical map. The functions $\alpha_+$ and $\mathrm{Tr}(\rho_i^+)$ are regular functions on $\fC_+$ and under our assumption that $G_0$ is simply-connected they define an isomorphism
\[\fC_+\cong A_{G_0}\times\bA^r\cong\bA^{2r}\]
Moreover, we have $k[V_{G_0}]^{G_0}\cong k[V_{T_0}]^W$ and hence
\[\fC_+\cong V_{T_0}//W\]

\subsubsection{Discriminant}
On $T$ (and hence $T_0$) we have the \emph{discriminant function}
\[\fD(t):=\prod_{\alpha\in\Phi}(1-\alpha(t))\]
which is clearly $W$-invariant and descends to a regular function on the Steinberg base $\fC=T//W$ (and hence its subspace $\fC_{G_0}=T_0//W$).\par 

The discriminant function $\fD$ on $T_0$ extends to a function $\fD_+$ on $T_0^+=(T_0\times T_0)/Z_0$ in the following way: 
\[\fD_+(t_1,t_2):=2\rho(t_1)\fD(t_2).\]
The extended discriminant $\fD_+$ is $W$-invariant and extends further to a regular function on $V_{T_0}$. Hence we view $\fD_+$ as an element in $k[V_{T_0}]^W$, i.e. a regular function on $\fC_+=V_{T_0}//W$.\par 
The zero locus of $\fD_+$ is a principal divisor on $\fC_+$, which we call the \emph{discriminant divisor}. The \emph{regular semisimple open subset} $\fC_+^{\mathrm{rs}}$ is the complement of the discriminant divisor in $\fC_+$. Similarly, on $\fC$ we have the regular semisimple locus $\fC^{\mathrm{rs}}$ defined as the complement of the principal divisor $\fD$. We denote $G^{\mathrm{rs}}:=\chi^{-1}(\fC^{\mathrm{rs}})$.
\begin{defn}\label{d-ga-defn}
For any $\ga\in G(F)^{\mathrm{rs}}$, we define its \emph{discriminant valuation} to be 
\[d(\ga):=\mathrm{val}(\fD(\chi(\ga))\in\bZ.\]
\end{defn}
Clearly $d(\ga)$ is (stable) conjugation invariant. To compute it, after conjugating by $G(\bar F)$ we may assume that $\ga\in T(\bar F)$, then we see that
\begin{equation}\label{d-ga-eq}
\begin{split}
d(\ga)&=\sum_{\alpha\in\Phi}\mathrm{val}(1-\alpha(\ga))\\
&=\mathrm{val}(\det(\mathrm{Id}-\mathrm{ad}_\ga:\fg(F)/\fg_\ga(F)\to\fg(F)/\fg_\ga(F))).
\end{split}
\end{equation}

\subsection{Steinberg sections and regular centralizers}
\subsubsection{Centralizer}
The homomorphism $\theta:G\to G_0^+$ in \eqref{G-ab-Cartesian-square-eq} induces an action of $G\times G$ on $G_0^+$ by left and right multiplication. Moreover, this action extends to $V_{G_0}$.
Define the centralizer group scheme $\cI$ over $V_{G_0}$ by
\[\cI=\{(g,\gamma)\in G\times V_{G_0} | g\gamma g^{-1}=\gamma\}\]
Let $V_{G_0}^{\mathrm{reg}}\subset V_{G_0}^0$ be the open subscheme consisting of elements $\ga\in V_{G_0}^0$ such that $\dim G_\ga=\mathrm{rank}(G)=r+\mathrm{rank}(G_{\mathrm{ab}})$.
Then $\cI|_{V_{G_0}^{\rm{reg}}}$ is smooth of relative dimension equal to $\mathrm{rank}(G)$ when restricted to $V_{G_0}^{\rm{reg}}$.\par 
Similarly, we can define a group scheme $\cI_0$ by considering the $G_0$ action on $V_{G_0}$ and we have
\[\cI=(\cI_0\times Z)/Z_0,\quad \cI_0=\cI\cap G_0.\]

\subsubsection{Coxeter elements}
Let $S=\{s_1,\dotsc,s_r\}$ be the set of simple reflections in $W$ corresponding to our choice of simple roots $\Delta$. Let $l:W\to\bN$ be the length function determined by $S$. For each $w\in W$, let $\mathrm{Supp}(w)\subset S$ be the subset consisting of those simple reflections which occurs in one (and hence every) reduced word expression of $w$. 
\begin{defn}\label{coxeter-definition}
An element $w\in W$ is called an \emph{$S$-Coxeter element} if it can be written as products of simple reflections in $S$, each occuring precisely once. In particular, $l(w)=r$ and $\mathrm{Supp}(w)=S$. Denote by $\mathrm{Cox}(W,S)$ the set of $S$-Coxeter elements in $W$. 
\end{defn}
In general, an element $w\in W$ is called a \emph{Coxeter element} if it is conjugate to an $S$-Coxeter element in $W$. 
\subsubsection{Steinberg section for $G_0$}
Recall that $G_0$ is simply connected. For each $S$-Coxeter element $w\in\mathrm{Cox}(W,S)$, and each choice of representatives $\dot{s}_i\in N_G(T)$ of the simple roots $s_i$, Steinberg defines a section $\epsilon^w_{G_0}:\fC_{G_0}\to G_0$ of the map $\chi_{G_0}:G_0\to\fC_{G_0}$. Moreover, it is shown that the equivalence class of $\epsilon_{G_0}^w$ depends neither on $w$ nor the choices $\dot{s}_i$, see \cite[7.5 and 7.8]{St65}. Here we say that two sections $\epsilon,\epsilon'$ are \emph{equivalent} if 
for all $a\in\fC_{G_0}$, $\epsilon(a)$ and $\epsilon'(a)$ are conjugate under $G_0$. 
\subsubsection{}
From the Cartesian diagram \eqref{G-ab-Cartesian-square-eq} we get a Cartesian diagram:
\[\xymatrix{
G\ar[d]_{\chi_G}\ar[r]^{\theta} & G_0^+\ar[d]^{\chi_{G_0^+}}\\
 \fC_G\cong \bA^r\times G_{\mathrm{ab}}\ar[r] & \bA^r\times T_0^{\mathrm{ad}}\cong\fC_{G_0^+}
}\]
where the bottom arrow is $\mathrm{Id}_{\bA^r}\times\theta_{\mathrm{ab}}$. Thus any section $\epsilon_+$ of $\chi_{G_0^+}$ will induce a section of $\chi_G$. If moreover, the image of $\epsilon_+$ is contained in $G_0^{+,\mathrm{reg}}$, then the image of $\epsilon$ is contained in $G^{\mathrm{reg}}$. In fact, we can even define a section for $\chi_+: V_{G_0}\to\fC_+$. 
\begin{prop}\label{extend-steinberg-section-prop}
For each Coxeter element $w\in\mathrm{Cox}(W,S)$ and choices of representatives $\dot{s_i}$ of the simple reflections $s_i$,
there exists a section $\epsilon_+^w:\fC_+\to V_{G_0}$ of the map $\chi_+: V_{G_0}\to \fC_+$. Moreover, the image of $\epsilon_+^w$ is contained in $V_{G_0}^{\mathrm{reg}}$.
\end{prop}
\begin{proof}
The first statement is \cite[Proposition 1.10 ]{Bou15}. The second statement is Proposition 1.16 in \textit{loc. cit.}
\end{proof}
For each $w\in\mathrm{Cox}(W,S)$, the equivalence class of the extended section $\epsilon_+^w$ is independant of the choice of representatives $\dot{s}_i$ of the simple reflections. However, for two \emph{different} $w,w'\in\mathrm{Cox}(W,S)$, the sections $\epsilon_+^w$ and $\epsilon_+^{w'}$ are not equivalent since, as we will see, $\epsilon_+^w(0)$ and $\epsilon_+^{w'}(0)$ are not conjugate.
\subsubsection{The nilpotent cone}
Let $\cN:=\chi_+^{-1}(0)$ be the \emph{nilpotent cone} in the Vinberg monoid $V_{G_0}$. Let $\cN^0:=\cN\cap V_{G_0}^0$ and $\cN^{\mathrm{reg}}:=\cN\cap V_{G_0}^{\mathrm{reg}}$ be the corresponding open subsets.
\begin{prop}\label{nilp-cone-prop}
For each $w\in\mathrm{Cox}(W,S)$, we have $\epsilon_+^w(0)\in\cN^{\mathrm{reg}}$ and there are bijections
\[\mathrm{Cox}(W,S)\xrightarrow{\sim}\mathrm{Irr}(\cN^{\mathrm{reg}})\xrightarrow{\sim}\mathrm{Irr}(\cN^0)\xrightarrow{\sim}\mathrm{Irr}(\cN)\]
induced by sending $w\in\mathrm{Cox}(W,S)$ to the component containing $\epsilon_+^w(0)$. Moreover, each irreducible component of $\cN^{\mathrm{reg}}$ is a single $G$-orbit (under the adjoint action) and $\cN^{\mathrm{reg}}$ is the disjoint union of open $G$-orbits in $\cN$
\end{prop}
\begin{proof}
The fact that $\epsilon_+^w(0)\in\cN^{\mathrm{reg}}$ is \citep[Lemme 1.17]{Bou15}. The other statements are \citep[Proposition 2.7,2.9,2.10]{Bou15} where they are deduced from results in \citep{He06} and \citep{Lu04}, \citep{Lu04b}.
\end{proof}
\begin{rem}
Recall that in the Lie algebra case, the fibers of the map $\fg^{\mathrm{reg}}\to\fc$ are irreducible and consist of a single adjoint $G$-orbit.
\end{rem}
\subsubsection{Universal centralizer}
It is proved in \cite{Bou15} that there exists a smooth commutative group scheme $\cJ_0$ over $\fC_+$ and a canonical homomorphism $\chi_+^*\cJ_0\to\cI_0$ which becomes an isomorphism when restricted to $V_{G_0}^{\mathrm{reg}}$. Moreover, for each $w\in\mathrm{Cox}(W,S)$, there is an isomorphism
\[(\epsilon_+^w)^*\cI_0\cong\cJ_0.\]
Define $\cJ:=(\cJ_0\times Z)/Z_0$. Then there is a canonical homomorphism $\chi_+^*\cJ\to\cI$ which is an isomorphism when restricted to $V_{G_0}^{\mathrm{reg}}$ and we also have an isomorpihsm
\[(\epsilon_+^w)^*\cI\cong\cJ\]
for each $w\in\mathrm{Cox}(W,S)$.\par 
Moreover, the morphism
\[\xymatrix@R=1pt{
c_w: G\times\fC_+\ar[r] & V_{G_0}^{\mathrm{reg}}\\
(g,a)\ar@{|->}[r] & g\epsilon_+^w(a)g^{-1}
}\]
factors through $(G\times\fC_+)/\cJ$ and induces an open embedding whose image we denote by $V_{G_0}^w$. It is clear that $V_{G_0}^w$ is stable under adjoint action of $G$. Moreover, it is shown in \cite{Bou15} that they form an open cover of 
$V_{G_0}^{\mathrm{reg}}$:
\begin{equation}\label{open-cover-V-reg-eq}
V_{G_0}^{\mathrm{reg}}=\bigcup_{w\in\mathrm{Cox}(W,S)}V_{G_0}^w
\end{equation}
see the proof of Proposition 2.12 in \textit{loc.cit.}
\begin{prop}\label{BJ-action-on-V_G-prop}
$\bB\cJ$ acts naturally on $[V_{G_0}/G]$. The action preserves the open substacks $[V_{G_0}^0/G]$, $[V_{G_0}^\mathrm{reg}/G]$ and $[V_{G_0}^w/G]$ for each $w\in\mathrm{Cox}(W,S)$. Moreover, the morphism
\[[\chi_+^w]: [V_{G_0}^w/G]\to\fC_+\]
induced by $\chi_+$ is a $\bB\cJ$ gerbe, neutralized by the extended Steinberg section $\epsilon_+^w$.
\end{prop}
The proof is the same as \cite[Proposition 2.2.1]{Ngo10}.
\begin{prop}
The number of irreducible components of the fibers of the map 
\[\chi_+^{\mathrm{reg}}:V_{G_0}^{\mathrm{reg}}\to\fC_+\]
is bounded above by $|\mathrm{Cox}(W,S)|$ and equality is achieved at $\cN^{\mathrm{reg}}=(\chi_+^{\mathrm{reg}})^{-1}(0)$.
\end{prop}
\begin{proof}
The first statement follows from \eqref{open-cover-V-reg-eq}. The second statement is in Proposition~\ref{nilp-cone-prop}.
\end{proof}
\begin{rem}
Consequently, unless all simple factors of $G_0$ are $\mathrm{SL}_2$, the action of $\bB\cJ$ on $[V_{G_0}^{\mathrm{reg}}/G]$ is not transitive. In other words, $[V_{G_0}^{\mathrm{reg}}/G]$ is not a $\bB\cJ$-gerbe, but rather a finite union of $\bB\cJ$ gerbes as in Proposition~\ref{BJ-action-on-V_G-prop}. This is different from Lie algebra situation, cf \citep[Proposition 2.2.1]{Ngo10}.
\end{rem}

\subsubsection{Galois description of $\cJ$}
Let $\prod\limits_{V_{T_0}/\fC_+}(T\times V_{T_0})$ be the restriction of scalar which associates to any $\fC_+$-scheme $S$ the set
\[\prod_{V_{T_0}/\fC_+}(T\times V_{T_0})(S)=\Hom_{V_{T_0}}(V_{T_0}\times_{\fC_+}S,T\times V_{T_0})\]
Then $W$ acts diagonally on $\prod\limits_{V_{T_0}/\fC_+}(T\times V_{T_0})$ and consider its fixed point subscheme
\[\cJ^1:=(\prod_{V_{T_0}/\fC_+}^{}T\times V_{T_0})^W.\]
The following is proved in \cite[Proposition 11]{Bou17}.
\begin{prop}\label{Galois-description-J-prop}
$\cJ^1$ is a smooth commutative group scheme over $\fC_+$. There exists a canonical homomorphism $\cJ\to\cJ^1$ which is an open embedding. In particular, we have a canonical isomorphism 
\[\mathrm{Lie}(\cJ)=\mathrm{Lie}(\cJ^1)\]
\end{prop}
In fact, the image of $\cJ$ in $\cJ^1$ can be described explicitly in \emph{loc. cit.}

\subsection{A special class of reductive monoids}\label{special-monoid-subsection}
We construct a special class of reductive monoids over $\cO$ from Vinberg's universal monoid. These monoids will be used to describe the generalized affine Springer fibers.
\subsubsection{}
Recall that $G_0$ is simply connected. Let $\lambda\in\Lambda^+$ be a dominant coweight of $G$. We define a monoid $V_G^\lambda$ over $\spec\cO$ by the following Cartesian diagram
\[\xymatrix{
V^\lambda_G\ar[r]\ar[d] & V_{G_0}\ar[d]^\alpha\\
\spec\cO\ar[r]^{-w_0(\lambda)} & A_{G_0}
}\]
where the bottom arrow is defined by the element 
\[(\varpi^{\langle-w_0(\lambda),\alpha_i\rangle})_{1\le i\le r}\in A_{G_0}(\cO)\cong\cO^r.\] 
For each $S$-Coxeter element $w\in\mathrm{Cox}(W,S)$, we define an open subshcemes $V_G^{\la,w}\subset V_G^{\la,\mathrm{reg}}\subset V_G^{\la,0}$ by pulling back $V_{G_0}^w\subset V_{G_0}^{\mathrm{reg}}\subset V_{G_0}^0$ along $-w_0(\la)$. By \eqref{open-cover-V-reg-eq}, $V_G^{\la,w}$ form an open cover of $V_G^{\la,\mathrm{reg}}$:
\begin{equation}\label{open-cover-V-lambda-reg-eq}
V_G^{\la,\mathrm{reg}}=\bigcup_{w\in\mathrm{Cox}(W,S)}
V_G^{\la,w}.
\end{equation}
\begin{rem}
The open subschemes $V_{G}^{\la,w}$ for $w\in\mathrm{Cox}(W,S)$ may not be mutually distinct. For example, if $\la=0$, then $V_G^\la=G_0$ and $V_G^{\la,w}=G_0^{\mathrm{reg}}$ for all $w$. In the other extreme, if $\la$ is dominant regular, in other words, $\langle\alpha_i,\la\rangle>0$ for all $1\le i\le r$, then $V_G^{\la,w}$ are mutually distinct.
\end{rem}
\subsubsection{}\label{C-plus-la-section}
Similarly, let $V_T^\lambda$ (resp. $\fC_+^\lambda$) be the pull back of $V_{T_0}$ (resp. $\fC_+$) under the map $-w_0(\lambda)$. Let 
\[\chi_\la: V_G^\la\to \fC_+^\la\]
be the base change of $\chi_+$. Then the canonical map 
\begin{equation}\label{cameral-cover-lambda-eq}
q_\la: V_T^\la\to\fC_+^\la
\end{equation}
is finite flat and generically a $W$-torsor. It is ramified along the divisor
\[\fD_+^\la:=(-w_0(\la))^*\fD_+\subset\fC_+^\la.\] 
\subsubsection{}\label{la-plus-section}
Let $\bar\la$ be the image of $\la$ in $X_*(T_0^{\mathrm{ad}})_+$ and consider the element
\[\varpi^{-w_0(\bar\la)}=(\varpi^{\langle-w_0(\la),\alpha_i\rangle})_{1\le i\le r}\in T_0^{\mathrm{ad}}(F)\subset A_{G_0}(F)\cong F^r.\]
By definition, we have
\[V_G^\la(F)=\{g\in G_0^+(F) | \alpha(g)=\varpi^{-w_0(\bar\la)}\}.\]
In particular, there is a canonical embedding $L^+V_G^\la\into LG_0^+$ and we would like to describe its image. For this we consider the element
\[\la^+:=(-w_0(\bar\la), w_0(\bar\la))\in X_*(T_0^+)\]
where we use description of $X_*(T_0^+)$ as in \eqref{T-plus-coweight-lattice-eq}. Then we have 
\[\varpi^{\la^+}=s(\varpi^{-w_0(\bar\la)})\in T_0^+(F)\]
 where $s$ is the canonical section of $\alpha$, cf. \eqref{section-alpha-eq}.

\begin{lem}\label{V-lambda-arc-space-lem}
With notations as above, we have
\[L^+V_G^{\la,0}=L^+G_0\varpi^{\la^+} L^+G_0\]
and
\[L^+V_G^\la = \bigcup_{\substack{\mu^+\in X_*(T_0^+)_+\\ \mu^+\le \la^+}}L^+G_0\varpi^{\mu^+}L^+G_0.\] 
\end{lem}

\begin{proof}
The action of $G_0\times G_0$ on $V_{G_0}^0$ is transitive on the fibres of the abelianisation map $\alpha$. Hence $L^+V_G^{\la,0}$ is a single $L^+G_0\times L^+G_0$-orbit. As the image of the section $s$ is contained in $V_{G_0}^0$, we have $\varpi^{\la^+}\in V_G^{\la,0}(\cO)$ and hence the first equality.\par 
Next we prove the second equality. It is clear that both sides are stable under $L^+G_0\times L^+G_0$. By Cartan decomposition for $G_0^+(F)$, it suffices to show that for all $\tilde\mu\in X_*(T_0^+)_+$ we have $\varpi^{\tilde\mu}\in V_G^\la(\cO)$ if and only if $\tilde\mu\le\la^+$.\par 
For each $1\le i\le r$, the representation $\rho_i^+$ of $G_0^+$ has highest weight $\omega_i^+:=(\omega^0_i,\omega^0_i)\in X^*(T_0^+)$, cf. \S\ref{omega-i-section}. Then for any $\tilde\mu\in X_*(T_0^+)_+$, the lowest order of entries of the matrix $\rho_i^+(\varpi^{\tilde\mu})$ equals to $\langle w_0(\omega^+_i),\tilde\mu\rangle$. Hence $\varpi^{\tilde\mu}\in V_G^\la(\cO)$ if and only if $\langle w_0(\omega^+_i),\tilde\mu\rangle\ge\langle w_0(\omega^+_i),\la^+\rangle$ for all $1\le i\le r$ and $\alpha(\varpi^{\tilde\mu})=\alpha(\varpi^{\la^+})$. This means precisely that $\tilde\mu\le\la^+$. 
\end{proof}

\subsubsection{} 
For each $w\in\mathrm{Cox}(W,S)$, the (equivalence class of) extended Steinberg section $\epsilon_+^w$ in Proposition~\ref{extend-steinberg-section-prop} defines an equivalence class of sections of $\chi_\la$:
\[\epsilon_\la^w: \fC_+^\la\to V_G^{\la,w}\subset V_G^{\la,\mathrm{reg}}.\]
Let $J^\la$ be the smooth commutative group scheme over $\fC_+^\la$ defined as the pull-back of $\cJ$. Then we have the following consequence of \ref{BJ-action-on-V_G-prop}:
\begin{prop}\label{BJ-gerbe-V-lambda-prop}
The morphism $[\chi_\la^w]:[V_G^{\la,w}/G]\to\fC_+^\la$ induced by $\chi_\la$ is a $\bB J^\la$-gerbe, neutralized by the section $\epsilon_\la^w$.
\end{prop}
\section{Generalized affine Springer fiber in the affine Grassmanian}\label{fiber-section}
Let $\lambda\in\Lambda_+$ be a dominant coweight. For any regular semisimple element $\ga\in G(F)^{\mathrm{rs}}$, we are interested in the following set
\[X_\ga^\la=\{g\in G(F)/G(\cO)| g^{-1}\ga g\in G(\cO)\varpi^\la G(\cO)\}\]
which we refer to as ``generalized affine Springer fiber". In this section we will first give criterions for nonemptiness of $X_\ga^\la$ and then equip $X_\ga^\la$ with the structure of ind-scheme and study its basic geometric properties.
\subsection{Nonemptiness}
We give two criterions for non-emptiness of the set $X_\ga^\la$.
\subsubsection{Newton Points}\label{Newton-point-section}
The first criterion uses the Newton point of $\ga$ which we now recall. For details, see \cite[\S4]{KoV}.\par 
Since $\ga\in G(F)^{\mathrm{rs}}$ is regular semisimple and the derived group $G_0$ is simply connected, the centralizer $G_\ga$ is a maximal torus defined over $F$. Choose an isomorphism $G_{\ga,\bar F}\cong T_{\bar F}$ over an algebraic closure $\bar F$. Then $\ga$ defines a $\bZ$-linear map
\[\xymatrix@R=1pt{
\mathrm{ev}_\ga: \check{\La}=X^*(T)\ar[r] & \bQ\\
\alpha\ar@{|->}[r] & \mathrm{val}(\alpha(\ga))
}\]
which can also be viewed as an element $\mathrm{ev}_\ga\in\Lambda_\bQ$. The $W$ orbit of $\mathrm{ev}_\ga$ is independant of the choice of the isomorphism $G_{\ga,\bar F}\cong T_{\bar F}$. We let $\nu_\ga\in\Lambda^+_\bQ$ be the unique element in the $W$-orbit of $\mathrm{ev}_\ga$ that lies in the dominant coweight cone and call it the \emph{Newton point of $\ga$}.\par 
In fact, $\nu_\ga$ actually lies in the subset
\[\La_{0,\bQ}^+\times X_*(G_{\mathrm{ab}})\subset\La_\bQ^+=\La_{0,\bQ}^+\times X_*(G_{\mathrm{ab}})_\bQ.\]
Using the Newton point, we get an alternative expression for the valuation of discriminant $d(\ga)$ (cf. Definition~\ref{d-ga-defn}) as follows:
\begin{lem}\label{d-ga-newton-lem}
Let $\ga\in G(F)^{\mathrm{rs}}$ and $\nu_\ga\in\La^+_\bQ$ its Newton point. After conjugation by $G(\bar F)$ we may assume that $\ga\in T(\bar F)^{\mathrm{rs}}$ and $\mathrm{val}(\alpha(\ga))\ge0$ for all positive root $\alpha$.  Then we have
\[d(\ga)=2\sum_{\alpha\in\Phi^+}\mathrm{val}(\alpha(\gamma)-1)-\langle2\rho,\nu_\ga\rangle.\]
\end{lem}
\begin{proof}
In \eqref{d-ga-eq}, separate the sum over $\Phi$ according to whether $\langle\alpha,\nu_\ga\rangle=0$ or not, then we get
\begin{equation}\label{d-ga-newton-eq}
d(\gamma)=\sum_{\substack{\alpha\in\Phi\\ \langle\alpha,\nu_\ga\rangle=0}}\mathrm{val}(\alpha(\gamma)-1)+
\sum_{\substack{\alpha\in\Phi\\ \langle\alpha,\nu_\ga\rangle<0}}\langle\alpha,\nu_\ga\rangle.
\end{equation}

By our assumption that $\mathrm{val}(\alpha(\ga))\ge0$ for $\alpha\in\Phi^+$, the first term in \eqref{d-ga-newton-eq} equals to 
\[2\sum_{\substack{\alpha\in\Phi^+\\ \langle\alpha,\nu_\ga\rangle=0}}\mathrm{val}(\alpha(\gamma)-1)=2\sum_{\alpha\in\Phi^+}\mathrm{val}(\alpha(\gamma)-1)\]
while the second term of \eqref{d-ga-newton-eq} equals to
\[\sum_{\alpha\in\Phi^-}\langle\alpha,\nu_\ga\rangle=-\sum_{\alpha\in\Phi^+}\langle\alpha,\nu_\ga\rangle=-\langle2\rho,\nu_\ga\rangle.\]
Hence the lemma follows.
\end{proof}

\subsubsection{The element $\ga_\la$}
Recall that the homomorphism $\theta:G\to G_0^+$ in \eqref{G-ab-Cartesian-square-eq} induces a homomorphism $\theta_*:X_*(T)\to X_*(T_0^+)$. Using the description \eqref{T-plus-coweight-lattice-eq}, we can write
\[\theta_*(\la)=(\la_0,\bar\la)\]
where $\bar\la$ is the image of $\la$ in $X_*(T_0^{\mathrm{ad}})$ and $\la_0\in X_*(T_0^{\mathrm{ad}})$ satisfies $\la_0+\bar\la\in X_*(T_0)$. Then we have $\theta_*(w_0(\la))=(\la_0,w_0(\bar\la))$ and in particular $\la_0+w_0(\bar\la)\in X_*(T_0)$. \par 
Consider the element
\begin{equation}\label{gamma-lambda-eq}
\ga_\la:=\varpi^{-\la_0-w_0(\bar\la)}\theta(\ga)\in G_0^+(F)
\end{equation}
where we view $\varpi^{-\la_0-w_0(\bar\la)}\in T_0(F)=Z^+(F)$ as an element in the center of $G_0^+(F)$.
\subsubsection{}
Suppose that $X_\ga^\la$ is nonempty. Take $g\in X_\ga^\la$, then $g^{-1}\ga g\in G(\cO)\varpi^\la G(\cO)$ and in particular
$\det(\ga)\in\det(\varpi^\la)G_{\mathrm{ab}}(\cO)$. Notice that multiplying $\ga$ by an element in $Z(\cO)$ does not change $X_\ga^\la$. Hence we may assume that $\det(\ga)=\det(\varpi^\la)$.

\begin{lem}\label{gamma-lambda-lem}
Under the assumption $\det(\ga)=\det(\varpi^\la)$, the element $\ga_\la$ belongs to $V_G^\la(F)$. In other words, 
\[\alpha(\ga_\la)=\varpi^{-w_0(\bar\la)}=(\varpi^{\langle-w_0(\la),\alpha_i\rangle})_{1\le i\le r}\in A_{G_0}(F).\]
\end{lem}
\begin{proof}
By the definition of the abelianization map $\alpha$, we have
\[\alpha(\theta(\varpi^\la))=\varpi^{\la_0}\]
By the commutative diagram \eqref{G-ab-Cartesian-square-eq} and the assumption that $\det(\ga)=\det(\varpi^\la)$, we get
\[\alpha(\theta(\ga))=\theta_{\mathrm{ab}}(\det(\ga))=\theta_{\mathrm{ab}}(\det(\varpi^\la))=\alpha(\theta(\varpi^\la))\] 
Therefore
\[\alpha(\ga_\la)=\varpi^{-\la_0-w_0(\bar\la)}\alpha(\theta(\ga))=\varpi^{-w_0(\bar\la)}.\]
\end{proof}
The following Proposition provides an alternative description of the set $X_\ga^\la$:

\begin{prop}\label{describe-X-ga-by-V-G-prop}
Assume that $\det(\ga)=\det(\varpi^\la)$, so in particular $\ga_\la\in V_G^\la(F)$ by \ref{gamma-lambda-lem}. Then 
\[X_\ga^\la=\{g\in G(F)/G(\cO)| \mathrm{Ad}(g)^{-1}(\ga_\la)\in V_G^{\la,0}(\cO)\}.\]
\end{prop}
\begin{proof}
For any $g\in G(F)/G(\cO)$, we have $g^{-1}\ga g\in G(\cO)\varpi^\la G(\cO)$ if and only if $g^{-1}\ga g\in G_0(\cO)\varpi^\la G_0(\cO)$ since $\det(\ga)=\det(\varpi^\la)$ by assumption. Note that $\theta(G_0(\cO)\varpi^\la G_0(\cO))$ is the $G_0(\cO)\times G_0(\cO)$ orbit of $\theta(\varpi^\la)=\varpi^{\theta(\la)}\in T_0^+(F)$.\par 
Since $\theta(\la)=(\la_0,\bar\la)$, we have
\[\theta(\la)+(-\la_0-w_0(\bar\la),0)=(-w_0(\bar\la),\bar\la)=w_0(\la^+)\in X_*(T_0^+)\]
where $\la^+$ is defined in \ref{la-plus-section}. Consequently
$g\in X_\ga^\la$ if and only if 
\[\mathrm{Ad}(g)^{-1}(\ga_\la)\in G_0(\cO)\varpi^{\la^+}G_0(\cO)=V_G^{\la,0}(\cO)\]
where the equality follows from \ref{V-lambda-arc-space-lem}.
\end{proof}

\subsubsection{}\label{C-le-la-section}
For each $\la\in\La^+$, we define the set
\[\fC_{\le\la}:=\chi(G(\cO)\varpi^\la G(\cO))\subset\fC(F)\]
where $\chi:G\to\fC$ is the Steinberg quotient. Under the isomorphism \eqref{Steinberg-base-eq} we have an identification
\[\fC_{\le\la}=\fC_+^\la(\cO)\times\det(\varpi^\la)G_{\mathrm{ab}}(\cO)\]
where $\fC_+^\la$ is defined in ~\ref{C-plus-la-section}.\par 
Now we can state the nonemptiness criterions.
\begin{prop}\label{nonempty-prop}
The following are equivalent:
\begin{enumerate}
\item $X_\ga^\lambda$ is nonempty;
\item $\nu_\ga\le_\bQ\la$, i.e. $\la-\nu_\ga$ is a $\bQ$-linear combinition of simple coroots with non-negative coefficients;
\item $\chi(\ga)\in\fC_{\le\la}$.
\end{enumerate}
\end{prop}
\begin{proof}
(1)$\Rightarrow$(2): This is \cite[Corollary 3.6]{KoV}.\par 

(2)$\Rightarrow$(3): By condition (2) we have $\det(\ga)\in\det(\varpi^\la) G_{\mathrm{ab}}(\cO)$. We may multiply $\ga$ by an element in $Z(\cO)$ and assume that $\det(\ga)=\det(\varpi^\la)$. By Lemma~\ref{gamma-lambda-lem}, we have $\ga_\la\in V_G^\la(F)$.\par
Let $F'/F$ be a finite extension of degree $e$ so that $\ga$ is split in $G(F')$.  Let $\varpi'=\varpi^{\frac{1}{e}}$ be a uniformizer of $F'$ and $\cO'=k[[\varpi']]\subset F'$ be the ring of integers. Then $e\cdot\nu_\ga\in\Lambda_+$ and $\ga$ is $G(F')$-conjugate to an element in $(\varpi')^{e\cdot\nu_\ga}T(\cO')$. Moreover, condition (2) implies that $e(\la-\nu_\ga)$ is an integral linear combination of simple coroots with non-negative coefficients and hence $\ga_\la$ is $G(F')$-conjugate to an element in $V_G^\la(\cO')$. Hence
\[\chi_\la(\ga_\la)\in \fC_+^\la(\cO')\cap\fC_+^\la(F)=\fC_+^\la(\cO)\]
which implies that
\[\chi(\ga)\in\fC_{\le\la}=\fC_+^\la(\cO)\times\det(\varpi^\la)G_{\mathrm{ab}}(\cO).\]  
(3)$\Rightarrow$(1): Since $\chi(\ga)\in\fC_{\le\la}$, in particular we have $\det(\ga)\in\det(\varpi^\la)G_{\mathrm{ab}}(\cO)$. After multiplying $\ga$ by an element in $Z(\cO)$ we may assume that 
$\det(\ga)=\det(\varpi^\la)$ and obtain the element $\ga_\la\in V_G^\la(F)$ as in Lemma~\ref{gamma-lambda-lem}.\par
Let $a=\chi_\la(\ga_\la)$. Then $a\in\fC_+^\la(\cO)$ by condition (3). Then for any $w\in\mathrm{Cox}(W,S)$ we have $\epsilon_\la^w(a)\in V_G^{\la,0}(\cO)$. Take $h\in G(F)$ such that $\mathrm{Ad}(h)^{-1}(\ga_\la)=\epsilon_\la^w(a)$. We see that $h\in X_\ga^\la$ by \ref{describe-X-ga-by-V-G-prop}. In particular, $X_\ga^\la$ is nonempty.
\end{proof}

\subsection{Ind-scheme structure}
We will equip the set $X_\ga^\la$ with an ind-scheme structure. We present two approaches, one based on the original definition, the other based on the description of $X_\ga^\la$ via the monoid $V_G^\la$ as in \ref{describe-X-ga-by-V-G-prop}. 
\subsubsection{}
Let $\mathrm{Gr}_G:=LG/L^+G$ be the affine Grassmanian for $G$, which is known to be an ind-projective ind-scheme over $k$. The positive loop group $L^+G$ acts by left multiplication on $\mathrm{Gr}_G$. Let $(LG)_\la:= L^+G\varpi^\la L^+G$ be the $k$-scheme whose set of $k$-points is the double coset $G(\cO)\varpi^\la G(\cO)$. 
\begin{defn}
The generalized affine Springer fiber $X_\ga^\la$ is the $k$-functor that associates to any $k$-algebra $R$ the set
\[X_\ga^\la(R)=\{g\in\mathrm{Gr}_G(R)| g^{-1}\ga g\in (LG)_\la(R)\}\]
\end{defn}
Then $X_\ga^\la$ is a locally closed sub-indscheme of $\mathrm{Gr}_G$.\par 
The following alternative moduli interpretation of $X_\ga^\la$ is a direct consequence of \ref{describe-X-ga-by-V-G-prop}.
\begin{lem}\label{second-def-lem}
For any $k$-algebra $R$, the set $X_\ga^\la(R)$ is equivalent to the discrete groupoid classifying pairs $(h,\iota)$ where $h$ is the horizontal arrow in the following commutative diagram
\[\xymatrix{
\spec R[[\varpi]] \ar[r]^h\ar[rd]_a & [V_G^{\la,0}/G]\ar[d]\\
 & \fC_+^\la
}\] 
and $\iota$ is an isomorphism between the restriction of $h$ to $\spec R((\varpi))$ and the composition
\[\spec R((\varpi))\xrightarrow{\ga} V_G^\la\to [V_G^\la/G].\]
\end{lem}
Replacing $V_G^{\la,0}$ by its open subschemes $V_G^{\la,\mathrm{reg}}$ (resp. $V_G^{\la,w}$) in the above diagram, we define the open subspaces $X_\ga^{\la,\mathrm{reg}}$ (resp. $X_\ga^{\la,w}$) of $X_\ga^\la$. By \eqref{open-cover-V-lambda-reg-eq}, we have an open cover
\begin{equation}\label{X-ga-reg-open-cover-eq}
X_\ga^{\la,\mathrm{reg}}=\bigcup_{w\in\mathrm{Cox}(W,S)}X_\ga^{\la,w}.
\end{equation}
\begin{rem}
The ind-scheme as defined above are non-reduced in general. However, we are only interested in their underlying reduced ind-scheme, which we use the same notations. 
\end{rem}
Later in \S\ref{finite-type-section} we will see that the ind-scheme $X_\ga^\la$ is actually a $k$-scheme locally of finite type under our assumption that $\ga$ is regular semisimple. 

\subsection{Symmetries}
Similar to the Lie algebra case, there are natural symmetries on the generalized affine Springer fiber $X_\ga^\la$. We recall the construction and emphasize the difference from the Lie algebra case.\par 
Assume $X_\ga^\la$ is nonempty and without loss of generality we also assume that $\det(\ga)=\det(\varpi^\la)$. Then by Proposition~\ref{nonempty-prop} we have
\[a=\chi_\la(\ga_\la)\in \fC_+^\la(F)^{\mathrm{rs}}\cap\fC_+^\la(\cO).\] 

\subsubsection{}
Let $J_a$ be the commutative group scheme over $\spec\cO$ obtained by pulling back $\cJ^\la$ along $a:\spec\cO\to\fC_+^\la$. Since $a\in\fC_+^\la(F)^{\mathrm{rs}}$ is generically regular semisimple, there is a canonical isomorphism 
$LJ_a\cong LG_\ga$ which allows us to identify the positive loop group $L^+J_a$ as a subgroup of $LG_\ga$. Consider the quotient group
\[\cP_a:=LJ_a/L^+J_a\cong LG_\ga/L^+J_a.\]
In other words, $\cP_a$ is the affine Grassmanian of $J_a$ classifying isomorphism classes of $J_a$-torsors on $\spec\cO$ with a trivialization of its restriction to $\spec F$. 

\subsubsection{}
The loop group $LG_\ga$ acts naturally on $X_\ga^\la$ and this action factors through $\cP_a$. Using the alternative modular interpretation of $X_\ga^\la$ in Lemma~\ref{second-def-lem}, the $\cP_a$ action is induced by the $\bB \cJ$ action on $[V_{G_0}^0/G]$ in Proposition~\ref{BJ-action-on-V_G-prop}. Moreover, the $\cP_a$ action on $X_\ga^\la$ induces an action on the open subspaces $X_\ga^{\lambda,\mathrm{reg}}$ and $X_\ga^{\la,w}$ for each $w\in\mathrm{Cox}(W,S)$.
\begin{prop}\label{X-ga-w-torsor-prop}
For each $w\in\mathrm{Cox}(W,S)$, $X_\ga^{\la,w}$ is a torsor under $\cP_a$.
\end{prop}
\begin{proof}
This is a consequence of \ref{BJ-gerbe-V-lambda-prop}.
\end{proof}
\begin{rem}
Unlike the Lie algebra case, $X_\ga^{\la,\mathrm{reg}}$ may not be a $\cP_a$-torsor in general. See the discussion in \S~\ref{regular-components-section}.
\end{rem}
\subsubsection{Galois description of $\cP_a$}
Let $A$ be the finite free $\cO$-algebra defined by the Cartesian diagram
\begin{equation}\label{cameral-cover-diagram-eq}
\xymatrix{
X_a:=\spec A\ar[r]\ar[d] & V_T^\la\ar[d]^{q_\la}\\
\spec\cO\ar[r]^a & \fC_+^\la
}
\end{equation}
Let $A^\flat$ be the normalization of $A$ and $X_a^\flat:=\spec A^\flat$. Then $W$ acts naturally on the $\cO$-algebras $A$ and $A^\flat$.\par 
Let $J_a^\flat$ be the finite type Neron model of $J_a$. Hence $J_a^\flat$ is a smooth commutative group scheme over $\cO$ such that $J_a^\flat(F)=J_a(F)=G_\ga(F)$ and $J_a^\flat(\cO)$ is the maximal bounded subgroup of $G_\ga(F)$. 

\begin{lem}\label{J_a-flat-Galois-description-lem}
There is a canocical isomorphism
\[J_a^\flat\cong(\prod_{A^\flat/\cO}T\times X_a^\flat)^W\]
\end{lem} 
\begin{proof}
The proof is the same as \cite[Proposition 3.8.2]{Ngo10}.
\end{proof}

\begin{cor}\label{Lie-P_a-cor}
$\mathrm{Lie}(\cP_a)=(\ft\otimes_k (A^\flat/A))^W$
\end{cor}
\begin{proof}
The quotient $L^+J_a^\flat/L^+J_a$ is an open subgroup of $\cP_a$. Hence we have isomorphism of $\cO$ modules
\[\mathrm{Lie}\cP_a\cong\mathrm{Lie}(L^+J_a^\flat)/\mathrm{Lie}(L^+J_a).\]
On the other hand, by \ref{Galois-description-J-prop}, we have
\[\mathrm{Lie}L^+J_a=(\ft\otimes_k A)^W\]
and by \ref{J_a-flat-Galois-description-lem},
\[\mathrm{Lie}L^+J_a^\flat=(\ft\otimes_k A^\flat)^W.\]
Hence the Corollary follows.
\end{proof}

\section{Geometry in the unramified case}\label{unr-section}
In this section we assume that $\gamma\in G(F)^{\mathrm{rs}}$ is an \emph{unramified} regular semisimple element. Since $k$ is algebraically closed,  after conjugation we may assume that $\gamma\in\varpi^\mu T(\cO)\cap G^{\mathrm{rs}}(F)$, where $\mu\in\Lambda_+$ is a \emph{dominant} coweight. In other words, $\mu=\nu_\ga$ is the Newton points of $\ga$. Then by Lemma~\ref{d-ga-newton-lem} the discriminant valuation for $\ga$ is
\[d(\ga)=2\sum_{\alpha\in\Phi^+}\mathrm{val}(\alpha(\gamma)-1)-\langle2\rho,\mu\rangle.\]
For later convenience, we also define
\begin{equation}\label{r-gamma-eq}
r(\gamma):=\sum_{\alpha\in\Phi^+}\mathrm{val}(\alpha(\gamma)-1)=\frac{1}{2}d(\gamma)+\langle\rho,\mu\rangle.
\end{equation} 
Fix a dominant coweight $\la\in\Lambda_+$ such that $\mu\le \la$. By Proposition~\ref{nonempty-prop}, this implies that $X_\ga^\la$ is nonempty.
\subsection{Admissible subsets of loop spaces}
\subsubsection{}
Let $LU$ and $L^+U$ be the loop space and arc space of $U$. For each integer $n\ge0$, let $U_n:=\ker(L^+U\to L^+_nU)$. Then $\{U_n\}_{n\ge0}$ form a decreasing sequence of compact open subgroups of $LU$.\par 
A subset of $L^+U$ is \emph{admissible} if it is the pre-image of a locally closed subset of $L^+_nU$ for some $n$. A subset $Z$ of $LU$ is \emph{admissible} of there exists a dominant  coweight $\mu_0$ such that $\varpi^{\mu_0}Z\varpi^{-\mu_0}$ is an admissible subset of $L^+U$.

\subsubsection{}
Consider the map
\begin{equation}\label{f-gamma-eq}
\xymatrix@R=1pt{
f_\gamma: LU\ar[r] & LU\\
u\ar@{|->}[r] & u^{-1}\gamma u\gamma^{-1}
}
\end{equation}
Since $\mu$ is dominant, we have $f_\gamma(U_n)\subset U_n$ for all $n\ge0$.\par 
Let $f_0: L^+U\to L^+U$  be the restriction of $f_\gamma$ to the arc space $L^+U$.
\begin{lem}\label{f_0-smooth-lem}
Let $V$ be an admissible subset of $L^+U$. Let $n$ be a positive integer such that 
\begin{itemize}
\item $V$ is right invariant under $U_n$ and
\item $n\ge\mathrm{val}(\alpha(\gamma)-1)$ for all $\alpha\in\Phi^+$.
\end{itemize}
Suppose moreover that $V\subset f_0(L^+U)$. Then the set $f_0^{-1}(V)$ is admissible and right invariant under $U_ n$. Moreover, $f_0$ induces a smooth surjective map 
\[f_0^{-1}(V)/U_n\to V/U_n\] whose fibers are isomorphic to $\bA^{r(\gamma)}$.
\end{lem}
\begin{proof}
Let $p_n: L^+U\to L_n^+U$ be the natural projection and 
\[\bar f_0: L_n^+U\to L_n^+U\]
the map induced by $f_0$. Since $V$ is right invariant under $U_n$ by assumption, a straightforward calculation shows that $f_0^{-1}(V)$ is also right invariant under $U_n$. Denote $\overline{V}:=V/U_n$. Then we have $f_0^{-1}(V)/U_n=\bar f_0^{-1}(\overline{V})$, a locally closed subset of $L_n^+U$. In particular, $f_0^{-1}(V)$ is admissible.\par  
Since $V\subset f_0(L^+U)$, the induced map $\bar f_0^{-1}(\overline{V})\to\overline{V}$ is surjective. It remains to show that it is smooth with fiber isomorphic to $\bA^{r(\ga)}$.\par 
 Choose an isomorphism (of schemes) $U\cong\bA^s$ where $s$ is the number of positive roots. For $u\in U(\cO/\varpi^n\cO)$, let $(u_\alpha)_{\alpha\in\Phi^+}$ be its coordinates with $u_\alpha\in\cO/\varpi^n\cO$. Then 
\[\bar f_0(u)=1\quad\Leftrightarrow\quad (\alpha(\gamma)-1)u_\alpha=0\in\cO/\varpi^n\cO.\]
Since $n\ge r_\alpha:=\mathrm{val}(\alpha(\gamma)-1)$, this condition is equivalent to $u_\alpha\in\varpi^{n-r_\alpha}\cO/\varpi^n\cO$. Such an element is determined by the first $r_\alpha$ coefficients (in $k$) of its Taylor expansion. Hence $\bar f_\gamma^{-1}(1)$ is an affine space of dimension
 \[\sum_{\alpha\in\Phi^+}r_\alpha=\sum_{\alpha\in\Phi^+}\mathrm{val}(\alpha(\gamma)-1)=r(\gamma)\]
by Definition~\ref{d-ga-defn} and Equation~\eqref{r-gamma-eq}.\par 
Finally, to see that the map is smooth, consider the right action of $L_n^+U$ on itself defined by $v*u:=u^{-1}v\ga u\ga^{-1}$ for $u,v\in U(\cO/\varpi^n\cO)$. Then $\bar f_0(u)=1*u$ and hence $\bar f_0$ is the orbit map at $1$ of the $L^+_nU$-action. Since the stabilizer subgroup $\bar f_\ga^{-1}(1)$ is smooth as we have just seen, the map $f_0^{-1}(V)/U_n\to V/U_n$ is smooth.
\end{proof}

\subsubsection{}
Let $\alpha_1,\dotsc,\alpha_s$ be an ordering of the set of positive roots such that $\mathrm{ht}(\alpha_i)\le\mathrm{ht}(\alpha_{i+1})$ . For each $1\le i\le s$, let $U[i]$ be the subgroup of $U$ generated by root groups $U_{\alpha_j}$ such that $j\ge i$. Also we denote $U[s+1]=1$. Then $U[1]=U$ and for each $1\le i\le s+1$, $U[i]$ is a normal subgroup of $U$ with successive quotients $U\langle i\rangle:=U[i]/U[i+1]\cong\bG_a$. The map $\mathrm{Ad}(\gamma)$ preserves each subgroup $U[i]$ and induces multiplication by $\alpha_i(\gamma)-1$ map on $U\langle i\rangle$. \par 
For each $1\le i\le s$, denote $r_i:=\mathrm{val}(\alpha_i(\gamma)-1)$. Then by Definition~\ref{d-ga-defn} and Equation~\eqref{r-gamma-eq},
\[\sum_{i=1}^s r_i=d(\gamma)+\langle2\rho,\mu\rangle=:r(\gamma).\]
The following lemma is analogous to \cite[Proposition 2.11.2 (2)]{GHKR}.
\begin{lem}\label{image-L^+U-lem}
Let $n$ be a positive integer such that $n\ge\mathrm{val}(\alpha(\gamma)-1)$ for every positive root $\alpha$. Then $U_n\subset f_\gamma(L^+U)$.
\end{lem}
\begin{proof}
Consider the right action of $L^+U$ on itself defined by $v*u:=u^{-1}v\gamma u\gamma^{-1}$ for every $u,v\in L^+U$.
Then $f_\gamma(u)=1*u$. Since $\gamma U_m\gamma^{-1}\subset U_m$, the action $*$ induces an action of $L_m^+U$ on itself for each $m\ge0$.\par 
Let $x\in U_n$. It suffices to find $u\in L^+U$ with $x*u=1$, for then $f_\gamma(u^{-1})=x$ and in particular $x\in f_\gamma(L^+U)$.\par 
First we find an element $u_1\in U(\cO)$ such that $x*u_1\in U[2](\cO)\cap U_n$.  Let $x_1\in U_{\alpha_1}(\cO)$ be the $\alpha_1$ component of $x$ under the canonical isomorphism $U=U[1]\cong U[2]\rtimes U_{\alpha_1}$. Fix an isomorphism $U_{\alpha_1}\cong\bG_a$ and identify $x_1$ with an element in $\varpi^n\cO$. Since $n\ge r_1:=\mathrm{val}(\alpha(\gamma)-1)$, we have $u_1:=(\alpha(\gamma)-1)^{-1}x_1\in\cO$, identified with an element of $U_{\alpha_1}(\cO)$. Then $x_1*u_1=1\in U_{\alpha_1}(\cO)$ which implies that $x*u_1\in U[2](\cO)$. Moreover since
\[\overline{x*u_1}=\bar x *\bar u_1=1*\bar u_1=\bar x_1*\bar u_1=\overline{x_1*u_1}=1\in U(\cO/\varpi^n\cO)\]
we get $x*u_1\in U[2](\cO)\cap U_n$.\par 
Next proceed in the similar way, we find $u_2\in L^+U$ such that $x*u_1*u_2\in U[3](\cO)\cap U_n$ and so on. In other words, we obtain a sequence $u_1,\dotsc,u_s\in U(\cO)$ such that $x*u_1*\dotsm*u_i\in U[i+1](\cO)\cap U_n$ for all $1\le i\le s$. Finally, let $u=u_1*\dotsm *u_s\in L^+U$ and we get $x*u=1$ as desired.
\end{proof}

The proof of the following lemma is inspired by \cite[Lemma 3.8]{KoV}.
\begin{lem}\label{inverse-image-U_n-lem}
For any $n\ge r(\gamma)$, we have $f_\gamma^{-1}(U_n)\subset U_{n-r(\gamma)}$.
\end{lem}
\begin{proof}
Let $u\in U(F)$ with $f_\gamma(u)\in U_n$. We will show by induction that 
\[u\in U[i](F)\cdot U_{n-\sum_{j<i}r_j}.\]
The case $i=1$ says $u\in U[1](F)=U(F)$ which is clear and the case $i=s+1$ gives the lemma since $\sum_{i=1}^sr_i=r(\ga)$ and $U[s+1]=1$.\par 
It remains to finish the induction step. By induction hypothesis we have $u=u_iv$ with $u_i\in U[i](F)$ and $v\in U_{n-\sum_{j<i}r_j}$. 
By assumption,
\[f_\gamma(u)=f_\gamma(u_iv)=v^{-1}\cdot u_i^{-1}\gamma u_i\gamma^{-1}\cdot \gamma v\gamma^{-1}\in U_n\] 
from which it follows that 
\[ u_i^{-1}\gamma u_i\gamma^{-1}\in U[i](F)\cap v\cdot U_n \cdot(\gamma v^{-1}\gamma^{-1})\subset U[i]_{n-\sum_{j<i}r_j}\]
After passing to the quotient $U\langle i\rangle$ we get
\[(\alpha_i(\gamma)-1)u_i\in U\langle i\rangle_{n-\sum_{j<i}r_j}\]
Under the isomorphism $U\langle i\rangle\cong\bG_a$, the set $U\langle i\rangle_{n-\sum_{j<i}r_j}$ is identified with $\varpi^{n-\sum_{j<i}r_j}\cO$. Moreover, since $\alpha_i(\gamma)-1\in\varpi^{r_i}\cO^\times$, we get that
\[u_i\in U[i+1](F)\cdot U_{n-\sum_{j<i+1}r_j}\]
and the same is true for $u=u_iv$. So we're done with the induction step.
\end{proof}

\begin{prop}\label{admissible-set-prop}
Let $Z$ be an admissible subset of the loop space $LU$. Then 
$f^{-1}_\gamma(Z)$ is admissible and there exists a positive integer $N$ such that for all $n\ge N$,  $f^{-1}_\gamma(Z)$ and $Z$ are right invariant under the group $U_n$
and the map 
\[f^{-1}_\gamma(Z)/U_n\to Z/U_n\] 
induced by $f_\gamma$ is smooth surjective whose geometric fibers are irreducible of dimension $r(\gamma)$.
\end{prop}
\begin{proof}
Let $n_0\ge r(\gamma)$ be a positive integer. In particular we have $n_0\ge\mathrm{val}(\alpha(\gamma)-1)$ for all positive roots $\alpha$. Choose  a dominant regular coweight $\mu_0$ such that 
\[Z^{\mu_0}:=\mathrm{Ad}(\varpi^{\mu_0})(Z)\subset U_{n_0}.\]
Then by Lemma~\ref{inverse-image-U_n-lem} we have 
\[f_\gamma^{-1}(Z^{\mu_0})\subset f_\gamma^{-1}(U_{n_0})\subset U_{n_0-r(\gamma)}\subset L^+U\]
Hence in particular 
\[\mathrm{Ad}(\varpi^{\mu_0})(f_\ga^{-1}(Z))=f_\gamma^{-1}(Z^{\mu_0})=f_0^{-1}(Z^{\mu_0})\]
Moreover, since $Z^{\mu_0}$ is an admissible subset of $L^+U$, $f_0^{-1}(Z^{\mu_0})$ is an admissible subset of $L^+U$ by Lemma~\ref{f_0-smooth-lem}. This shows that $f_\ga^{-1}(Z)$ is admissible.\par 
Let $m>n_0$ be a positive integer such that  $Z^{\mu_0}$ and $f_\gamma^{-1}(Z^{\mu_0})$ are invariant under right multiplication by $U_m$. For all $n\ge m$, since the map $f_\gamma$ commutes with conjugation by $\varpi^{\mu_0}$, $Z$ and $f_\gamma^{-1}(Z)$ are right invariant under the group $U_n^{-\mu_0}:=\varpi^{-\mu_0}U_n\varpi^{\mu_0}$. Then we get the following commutative diagram
\[\xymatrix{
f_\gamma^{-1}(Z)/U_n^{-\mu_0}\ar[r]\ar[d]^{\simeq} & Z/U_n^{-\mu_0}\ar[d]^{\simeq}\\
f_\gamma^{-1}(Z^{\mu_0})/U_n\ar[r] & Z^{\mu_0}/U_n
}\]
where the horizontal arrows are induced by $f_\gamma$ and the vertical arrows are isomorphisms induced by $\mathrm{Ad}(\varpi^{\mu_0})$.\par 
By Lemma~\ref{image-L^+U-lem}, $Z^{\mu_0}\subset U_{n_0}\subset f_\gamma(L^+U)$. Therefore we can apply Lemma~\ref{f_0-smooth-lem} to conclude that the lower horizontal map is surjective smooth whose fibers are isomorphic to $\bA^{r(\gamma)}$. Hence the same is true for the upper horizontal map.\par 
Let $N$ be a positie integer such that for all $n\ge N$, $U_n\supset U_{n'}^{-\mu_0}$ for some $n'\ge m$. Consider the following diagram
\[\xymatrix{
f_\ga^{-1}(Z)/U_{n'}^{-\mu_0}\ar[r]\ar[d] & Z/U_{n'}^{-\mu_0}\ar[d]\\
f_\ga^{-1}(Z)/U_n\ar[r] & Z/U_n
}\]
The two vertical maps are smooth surjective with fibers isomorphic to the irreducible scheme $U_n/U_{n'}^{-\mu_0}$ and the upper horizontal map is smooth surjective with fibers isomorphic to $\bA^{r(\ga)}$ as we have just seen. Hence the lower horizontal map is smooth surjective with irreducible fibers of dimension $r(\ga)$.
\end{proof}

\subsection{A study of $Y_\ga^\la$}

\subsubsection{}\label{Y-gamma-section}
Let $Y_\ga^\la$ be the $k$-functor that associates to any $k$-agebra $R$ the set
\[Y_\ga^\la(R)=\{u\in LU/L^+U (R)| u^{-1}\ga u\in (LG)_\la(R)\}\]
Then $Y_\ga^\la$ is a locally closed sub-indscheme of $X_\ga^\la$ and we also denote its underlying reduced  structure by the same symbol. In particular, its set of $k$-points is
\[Y_\ga^\la(k)=\{u\in U(F)/U(\cO)| u^{-1}\ga u\in G(\cO)\varpi^\la G(\cO)\}.\]
\subsubsection{Relation with MV-cycles}
To understand the structure of $Y_\gamma^\la$, we recall the map \eqref{f-gamma-eq} 
\[f_\gamma: LU\to LU\]
defined by $f_\gamma(u)=u^{-1}\gamma u\gamma^{-1}$. In the following, to ease notation, we let $K:=L^+G$. Then we have
\[Y_\gamma^\lambda=(f_\gamma^{-1}(K \varpi^\lambda K\varpi^{-\mu}\cap LU)/L^+U\]
Recall the Mirkovic-Vilonen cycles in the affine Grassmanian:
\[S_\mu\cap\mathrm{Gr}_\la=(LU\varpi^\mu K\cap K\varpi^\lambda K)/K\]
From this description we get an isomorphism
\begin{equation}\label{MV-cycle-presentation-equation}
\xymatrix@R=1pt{
(LU\cap K\varpi^\lambda K\varpi^{-\mu})/\varpi^\mu L^+U\varpi^{-\mu}\ar[r] &
S_\mu\cap\mathrm{Gr}_\lambda\\
u\ar@{|->}[r] & u\varpi^\mu
}\end{equation}
In summary, we have the following diagram
\[\xymatrix{
f_\gamma^{-1}(K\varpi^\lambda K\varpi^{-\mu}\cap LU)\ar[r]^{f_\gamma}\ar[d] & K\varpi^\lambda K\varpi^{-\mu}\cap LU\ar[d]\\
Y_\gamma^\lambda & S_\mu\cap\mathrm{Gr}_\lambda
}\]
where the left vertical arrow is an $L^+U$-torsor and the right vertical arrow is a torsor under the group $\varpi^\mu L^+U\varpi^{-\mu}$.\par 
\begin{thm}\label{Y-gamma-thm}
 $Y_\ga^\la$ is an equi-dimensional quasi-projective variety of dimension $\langle\rho,\la\rangle+\frac{1}{2}d(\ga)$, where $d(\ga)$ is the discriminant valuation, cf. Definition~\ref{d-ga-defn}. Moreover, the number of irreducible components of $Y_\ga^\la$ equals to $m_{\la\mu}$, the dimension of $\mu$-weight space in the irreducible representation $V_\la$ of $\hat G$ with highest weight $\la$. 
\end{thm}
\begin{proof}
Apply Proposition~\ref{admissible-set-prop} to the admissible subset $Z=K\varpi^\lambda K\varpi^{-\mu}\cap LU$ of $LU$, we see that there exists a large enough positive integer $n$ such that in the following diagram
\[\xymatrix{
f_\gamma^{-1}(K\varpi^\lambda K\varpi^{-\mu}\cap LU)/U_n\ar[r]^{\bar f_\gamma}\ar[d] & (K\varpi^\lambda K\varpi^{-\mu}\cap LU)/U_n\ar[d]\\
Y_\gamma^\lambda & S_\mu\cap\mathrm{Gr}_\lambda
}\]
\begin{enumerate}
\item All schemes are of finite type;
\item The map $\bar f_\gamma$ induced by $f_\gamma$ is smooth surjective whose geometric fibers are irreducible of dimension $r(\gamma)$, where we recall that $r(\ga)$ is defined in \eqref{r-gamma-eq};
\item $U_n$ is contained in $\varpi^\mu L^+U\varpi^{-\mu}$, hence also $L^+U$;
\item The left vertical map is smooth surjective with fibers isomorphic to the irreducible scheme $L^+U/U_n$;
\item The right vertical map is smooth with fibers isomorphic to the irreducible scheme $\varpi^\mu L^+U\varpi^{-\mu}/U_n$.
\end{enumerate}
Since $Y_\ga^\la$ is of finite type, it is a locally closed subscheme of a closed Schubert variety. In particular, $Y_\ga^\la$ is quasi-projective since closed Schubert varieties are projective.\par 
Since the MV-cycle $S_\mu\cap\mathrm{Gr}_\la$ is equidimensional of dimension $\langle\rho,\la+\mu\rangle$. Hence by (2)-(5) we see that $Y_\ga^\la$ is equidimensional of dimension  
\begin{equation}
\begin{split}
\dim Y_\ga^\la&=\dim(S_\mu\cap\mathrm{Gr}_\la)+\dim\varpi^\mu U(\cO)\varpi^{-\mu}/U_n^{-\lambda_0}+r(\gamma)-\dim U(\cO)/U_n^{-\lambda_0}\\
&=\langle\rho,\la+\mu\rangle-\langle2\rho,\mu\rangle+r(\gamma)=\langle\rho,\la\rangle+\frac{1}{2}d(\gamma)
\end{split}
\end{equation}

Moreover, by \cite[\href{http://stacks.math.columbia.edu/tag/037A}{Tag 037A}]{stacks-project} the 3 maps in the diagram above induces a canonical bijections between set of irreducible components
\[\mathrm{Irr}(Y_\ga^\la)\xrightarrow{\sim}\mathrm{Irr}(S_\mu\cap\mathrm{Gr}_\la).\]
Hence the number of irreducible components of $Y_\ga^\la$ equals to the number of irreducible components of the MV-cycle $S_\mu\cap\mathrm{Gr}_\la$, which is known to be $m_{\la\mu}$.
\end{proof}

\begin{cor}\label{dim-unr-cor}
Recall that $\ga\in\varpi^\mu T(\cO)\cap G^{\mathrm{rs}}(F)$, $\mu\in\Lambda$ and $\mu\le\la$. Then $X_\ga^\la$ is a scheme locally of finite type, equidimensional of dimension 
\[\dim X_\ga^\la=\langle\rho,\la\rangle+\frac{1}{2}d(\ga).\]
Moreover, the number of $G_\ga(F)$-orbits on its set of irreducible component $\mathrm{Irr}(X_\ga^\la)$ equals to $m_{\la\mu}$.
\end{cor}
\begin{proof}
There is a natural morphism
\[\xymatrix@R=1pt{
Y_\gamma^\lambda\times \Lambda\ar[r] & X_\gamma^\la \\
(u,\nu)\ar@{|->}[r] & u\varpi^\nu
}\]
which induces bijection on $k$-points and a stratification of $X_\ga^\la$ such that each strata is isomorphic to $Y_\ga^\la$. Thus $X_\ga^\la$ is a scheme locally of finite type and the assertions about equidimensionality and dimension formula follows from the corresponding statements for $Y_\ga^\la$.\par 
The $LG_\ga$ action on the set $\mathrm{Irr}(X_\ga^\la)$ factors through $\pi_0(\cP_a)=\Lambda$ and hence $LG_\ga$-orbits on $\mathrm{Irr}(X_\ga^\la)$ corresponds bijectively to the set $\mathrm{Irr}(Y_\ga^\la)$. Thus the number of orbits equals to $m_{\la\mu}$. 
\end{proof}

\section{Geometry in general}\label{general-case-section}
In this section we let $\ga\in G(F)^{\mathrm{rs}}$ be any regular semisimple element and $\la\in\Lambda^+$. Assume without loss of generality that $X_\ga^\la$ is nonempty and $\det(\ga)=\det(\varpi^\la)$. Then we get an element $\ga_\la\in V_G^\la(F)$ as in \ref{gamma-lambda-lem}. Moreover, the Newton point of $\ga$ satisfies $\nu_\ga\le_\bQ\la$ and $\chi(\ga)\in\fC_{\le\la}$ by Proposition~\ref{nonempty-prop}. 
\subsection{Finiteness properties}\label{finite-type-section}
 We show in this section that $X_\ga^\la$, a'priori an ind-scheme, is actually a scheme locally of finite type. This has already been proved for unramified conjugacy classes in Corollary~\ref{dim-unr-cor}. It remains to reduce the general case to the unramified case. This reduction step is completely analogous to the Lie algebra case. For the reader's convenience, we include the details. We follow the exposition in \cite[\S 2.5]{Yun15}. See also \citep{Bou15}.
\subsubsection{Splitting $\gamma$}
Let $F'/F$ be a finite extension of degree $e$ so that $\gamma$ splits over $F'$.  Let $\varpi'=\varpi^{1/e}\in F'$ be a uniformizer and $\cO'=k[[\varpi']]$ the ring of integers in $F'$. Let $\sigma$ be a generator of the cyclic group $\mathrm{Gal}(F'/F)$\par 
Choose $h\in G(F')$ such that $\mathrm
{Ad}(h) G_\ga =T$. Then $h\sigma(h)^{-1}\in N_G(T)(F')$ and we let $w\in W$ be its image.\par 
Consider the embedding
\[\xymatrix@R=1pt{
\iota_\ga:\La\ar[r] & G_\ga(F')\\
\mu\ar@{|->}[r] & \mathrm{Ad}(h)^{-1}\varpi^\mu
}\]
Let $\La_\ga:=\iota_\ga^{-1}(G_\ga(F))$. It follows immediately that $\La_\ga\subset\La^w$ where $\La^w$ is the fixed point set of $w$ on $\La$. Moreover, $\La_\ga$ can be identified with the coweight lattice of the maximal $F$-split subtorus of $G_\ga$. In particular, $(\La_\ga)_\bQ=(\La^w)_\bQ$ so that $\La_\ga\subset\La^w$ is a subgroup of finite index. 
\begin{prop}
There exists a closed subscheme $Z\subset X_\ga^\la$ which is projective over $k$ such that $X_\ga^\la=\cup_{\ell\in\La_\ga}\ell\cdot Z$. Here $\ell\in\Lambda_\ga$ acts on $X_\ga^\la$ via the embedding $\iota_\ga$.
\end{prop}
\begin{proof}
We reprase the argument in \cite[\S 2.5.7]{Yun15}. Let $\widetilde{X}_\ga^{e\la}$ be the generalized affine Springer fiber of coweight $e\la$ for $\ga$ in $\mathrm{Gr}_{G_{F'}}$, the affine Grassmanian of $G_{F'}$. Then $\sigma$ acts naturally on $\widetilde{X}_\ga^{e\la}$ and the fixed points sub-indscheme $(\widetilde{X}_\ga^{e\la})^\sigma$ contains $X_\ga^\la$ (but they are not equal in general). 
Let $\ga'=h\ga h^{-1}\in T(F')$ and $\widetilde{X}_{\ga'}^{e\la}$ the corresponding generalized affine Springer fiber in $\mathrm{Gr}_{G_{F'}}$. Then 
\[\widetilde{X}_{\ga'}^{e\la}=h\cdot \widetilde{X}_\ga^{e\la}\]
By Theorem~\ref{Y-gamma-thm},
there is a locally closed subscheme $\widetilde{Y}_{\ga'}^{e\la}$ of $\widetilde{X}_{\ga'}^{e\la}$ such that 
\[\widetilde{X}_{\ga'}^{e\la}=\cup_{\ell\in\La}\ell\cdot Y_{\ga'}^{e\la}.\] 
Let $\widetilde{Z}$ be the closure of $h^{-1}\widetilde{Y}_{\ga'}^{e\la}$ in $\widetilde{X}_\ga^{e\la}$. Then $\widetilde{Z}$ is projective over $k$ and $\widetilde{X}_\ga^{e\la}=\cup_{\ell\in\La}\ell\cdot\widetilde{Z}$.\par 
 Recall that $w\in W$ is represented by $h\sigma(h)^{-1}$. One can check that $\sigma(\widetilde Z)=\widetilde Z$ and more generally $\sigma(\ell\cdot\widetilde{Z})=w(\ell)\cdot\widetilde{Z}$ for all $\ell\in\La$.Consequently,
\[(\widetilde{X}_\ga^{e\la})^\sigma=\cup_{\ell\in\La^w}\ell\cdot\widetilde{Z}=\cup_{\ell\in\La_\ga}\ell\cdot (C\cdot\widetilde{Z})\]
where $C\subset\La^w$ is a finite set of representatives of the quotient $\La^w/\La_\ga$. Hence, $C\cdot\widetilde{Z}$ is a finite type scheme.\par 
Finally let $Z:=(C\cdot\widetilde{Z})\cap X_\ga^\la$. Then $Z$ is a finite type subscheme of $X_\ga^\la$. Hence $Z$ is projective over $k$ and $X_\ga^\la=\cup_{\ell\in\La_\ga}\ell\cdot Z$.
\end{proof}

As a consequence, we immediately get:
\begin{thm}\label{finite-type-thm}
The ind-scheme $X_\ga^\la$ is a finite dimensional $k$-scheme, locally of finite type. Moreover, the lattice $\La_\ga$ acts freely on $X_\ga^\la$ and the quotient $X_\ga^\la/\La_\ga$ is representable by a proper algebraic space over $k$. 
\end{thm}

\subsection{Dimension of the regular locus}
Recall that the regular locus $X_\ga^{\la,\mathrm{reg}}$ is an open subscheme of $X_\ga^\la$ on which the action of $\cP_a=LG_\ga/L^+J_a$ is free (but not necessarily transitive). 
\begin{thm}\label{dim-reg-locus-thm}
\[\dim\cP_a=\dim X_\ga^{\la,\mathrm{reg}}=\langle\rho,\la\rangle+\frac{d(\ga)-c(\ga)}{2}\]
where 
\begin{itemize}
\item $d(\ga):=\mathrm{val}(\det(\mathrm{Id}-\mathrm{ad}(\ga):\mathfrak{g}(F)/\mathrm{g}_\ga(F)\to \mathfrak{g}(F)/\mathrm{g}_\ga(F)))$.
\item $c(\ga):=\mathrm{rank}(G)-\mathrm{rank}_FG_\ga$, where $\mathrm{rank}_FG_\ga$ is the dimension of the maximal $F$-split subtorus of $G_\ga$.
\end{itemize}
Moreover, $X_\ga^{\la,\mathrm{reg}}$ is equidimensional.
\end{thm}
\begin{proof}
The first equality follows from the fact that the $\cP_a$-orbits in $X_\ga^{\la,\mathrm{reg}}$ are open and the action is free.\par 
When $\ga$ is unramified (hence split as $k$ is algebraically closed), the second equality follows from Corollary~\ref{dim-unr-cor}. It remains to reduce to this case. The argument is similar to that of Bezrukavnikov's in Lie algebra case, cf. \cite{Be96}, which we reformulate using the Galois description of universal centralizer.\par 
Let $A$ be the finite free $\cO$-algebra defined by the Cartesian diagram \eqref{cameral-cover-diagram-eq} and $A^\flat$ the normalization of $A$. Then $W$ acts naturally on the $\cO$-algebras $A$ and $A^\flat$ and by \ref{Lie-P_a-cor}, we get
\[\dim\cP_a=\dim_k(\ft\otimes_k (A^\flat/A))^W.\]
Let $\tilde F/F$ be a ramified extension of degree $e$, with ring of integers $\tilde\cO=k[[\varpi^{\frac{1}{e}}]]$, such that $\ga$ is split over $\tilde F$. Let $\sigma$ be a generater of the cyclic group $\Gamma:=\mathrm{Gal}(F'/F)$. Let $\tilde A:=A\otimes_\cO\tilde\cO$ and $\tilde A^\flat$ its normalization. We remark that $\tilde A^\flat$ is not the same as $A^\flat\otimes_\cO\tilde\cO$ in general.  Let $\tilde\cP_a=LG_{\ga,F'}/L^+J_{a,F'}$. Then by the dimension formula in split case, we have
\[\dim_k(\ft\otimes_k\tilde A^\flat/\tilde A)^W=\dim\tilde\cP_a=\langle\rho,e\la\rangle+\frac{1}{2}e\cdot d(\ga).\]
As $\ga$ split over $\tilde\cO$, we have 
\[\tilde A^\flat\cong\tilde\cO[W]:=\tilde\cO\otimes_k k[W]\] as $W$-module. Here $W$ acts on $\tilde\cO[W]$ via right regular representation. Moreover, there exists an element $w_\ga\in W$ of order $e$ such that under the above isomorphism, the natural action of $\sigma\in\Gamma$ on $\tilde A^\flat$ becomes $\sigma\otimes l_{w_\ga}$ where $l_{w_\ga}$ denotes the left regular action of $w_\ga$ on $k[W]$. In particular, the action of $W$ and $\Gamma$ commutes with each other. With these considerations, we obtain an isomorphism
\[(\ft\otimes_k\tilde A^\flat)^W\cong\ft\otimes_k\tilde\cO\]
which intertwines the action of $\sigma\in\Gamma$ on the left hand side with the action of $w\otimes\sigma$ on the right hand side. \par 
Moreover, we have an equality
\[(\ft\otimes_k \tilde{A}^\flat)^\Gamma=\ft\otimes_k A^\flat\]
which remains true after taking $W$-invariants since the $\Gamma$ action commutes with $W$ action. In particular, we have
\[M:=(\ft\otimes_k\tilde\cO)^\Gamma=(\ft\otimes_k A^\flat)^W\]
Moreover, it is clear that from the definition of $W$ action that
\[(\ft\otimes_k\tilde A)^W=(\ft\otimes_k A)^W \otimes_\cO \tilde\cO.\]
Thus we get
\begin{equation}\label{P-a-ramified-case-calculation-eq}
\begin{split}
\dim\cP_a & =\dim_k(\ft\otimes_k A^\flat/A)^W=
\frac{1}{e}\dim_k(\ft\otimes_k (A^\flat/A)\otimes_\cO \tilde\cO)^W=\frac{1}{e}\dim_k \left(\frac{M\otimes_\cO \tilde \cO}{(\ft\otimes_k\tilde A)^W}\right)\\
&=\langle\rho,\la\rangle+\frac{1}{2}d(\ga)-\frac{1}{e}\dim_k\left(\frac{\ft\otimes_k\tilde\cO}{M\otimes_\cO\tilde\cO}\right)
\end{split}
\end{equation}
Since the element $w_\ga\in W$ has order $e$, its eigenvalues are $e$-th roots of unit. Let $\zeta$ be a primitive $e$-th root of unit and $\ft(i)$ the subspace of $\ft$ on which $w_\ga$ acts via the scalar $\zeta^i$. In particular, $\ft(0)=\ft^{w_\ga}$ is the $w_\ga$ invariant subspace. Then we have
\[M:=(\ft\otimes_k\tilde\cO)^\Gamma=\bigoplus_{i=0}^{e-1}\ft(i)\otimes_k \varpi^{\frac{e-i}{e}}\] 
The existence of a $W$-invariant nondegenerate symmetric bilinear form on $\ft$ gaurantees that $\dim_k\ft(i)=\dim_k\ft(e-i)$, from this we obtain that
\[\dim_k\left(\frac{\ft\otimes_k\tilde\cO}{M\otimes_\cO \tilde\cO} \right)=e(\dim_k\ft-\dim_k\ft^{w_\ga})= e\cdot c(\ga)\]
Combined with \eqref{P-a-ramified-case-calculation-eq}, we obtain
\[\dim\cP_a=\langle\rho,\la\rangle+\frac{1}{2}(d(\ga)-c(\ga)).\]
Finally, $X_\ga^{\la,\mathrm{reg}}$ is equidimensional since it is a finite union of $\cP_a$-torsors.
\end{proof} 

\subsection{Some 0-dimensional generalized affine Springer fibers}
We study an important class of $X_\ga^\la$ that are 0 dimensional. The results in this subsection will be used in \citep{BC17}. We keep the assumptions at the beginning of \S\ref{general-case-section}.
 
\subsubsection{$\la$-twisted discriminant valuation}
Let $a:=\chi_\la(\ga_\la)\in\fC_+^\la$ and view it as a map $a:\spec\cO\to\fC_+^\la$. Recall the discriminant divisor $\fD_\la\subset\fC_+^\la$ defined in \S\ref{C-plus-la-section}. We define the \emph{$\la$-twisted discriminant valuation} to be
\[d_\la(a):=\mathrm{val}(a^*\fD_\la)\in\bZ\] 
Combining the definition of $\fD_\la$ with Lemma~\ref{gamma-lambda-lem} and equation \eqref{d-ga-newton-eq} we get
\begin{equation}\label{d-la-eq}
\begin{split}
d_\la(a)&=2\cdot\mathrm{val}(\rho(\alpha(\ga_\la)))+d(\ga)\\
&=\langle2\rho,\la\rangle+d(\ga)\\
&=\sum_{\substack{\alpha\in\Phi\\ \langle\alpha,\nu_\ga\rangle=0}}\mathrm{val}(\alpha(\ga)-1)+\langle 2\rho, \la-\nu_\ga\rangle
\end{split}
\end{equation}

\begin{prop}
Suppose $d_\la(a)=0$. Then $\ga$ is unramified (or split) and $\mathrm{dim}X_\ga^\la=0$. Moreover, $X_\ga^\la=X_\ga^{\la,\mathrm{reg}}$ and it is a torsor under $\cP_a$.
\end{prop}
\begin{proof}
By assumption $\nu_\ga\le_\bQ\la$, so the two terms in \ref{d-la-eq} are both non-negative. The condition $d_\la(a)=0$ then implies that the two terms are both $0$.\par 
In particular, $\nu_\ga=\la$. Let $I(\la)=\{\alpha\in\Delta, \langle\alpha,\la\rangle=0\}$ and $L$ be the Levi subgroup of $G$ whose root system has simple reflections $I(\la)$. Then $L$ is the centralizer of the coweight $\la=\nu_\ga\in\La^+$ and hence $\ga\in L(F)$ by \citep[Corollary 2.4]{KoV}.   Apply Proposition~\ref{nonempty-prop} to the group $L$, we see that $\ga\in\varpi^\la L_0(\cO)$ where $L_0$ is the derived group of $L$. Write $\ga=\varpi^\la\ga'$ with $\ga'\in L_0(\cO)$. Let $d_L(\cdot)$ be the discriminant valuation for the group $L$. Then by \eqref{d-ga-eq} and the vanishing of $d_\la(a)$, we have
\[d_L(\ga')=d_L(\ga)=\sum_{\substack{\alpha\in\Phi\\ \langle\alpha,\la\rangle=0}}\mathrm{val}(\alpha(\ga')-1)=0\]
Let $a':=\chi_{L_0}(\ga')\in\fC_{L_0}(\cO)$. Consider the following Cartesian diagram
\[\xymatrix{
\chi_{L_0}^{-1}(a')\ar[r]\ar[d] & T\cap L_0\ar[d]^{q_{L_0}}\\
\spec\cO\ar[r]^{a'} & \fC_{L_0}
}\]
As $d_L(\ga')=0$, the image of $a'$ lands in $\fC_{L_0}^{\mathrm{rs}}$ over which $q_{L_0}$ is \'etale (recall that $L_0$ is simply connected since $G_0$ is). Hence $\chi_{L_0}^{-1}(a')$ is a trivial \'etale cover of $\spec\cO$. This implies that $\ga'$ is $L_0(F)$-conjugate to an element in $T(\cO)\cap L_0(\cO)$ and hence in particular $\ga$ is unramified in $G(F)$.\par 
Finally we apply Corollary~\ref{dim-unr-cor} to the unramified conjugacy class $\ga$ with $\nu_\ga=\la$. We see that $\dim X_\ga^\la=0$ and $X_\ga^\la$ is a $\cP_a$-torsor since the weight multiplicity $m_{\la\la}=1$. In particular $X_\ga^\la= X_\ga^{\la,\mathrm{reg}}$.
\end{proof}

\subsection{Irreducible components}\label{irr-components-conjecture-section}
\subsubsection{Stratification of dominant weight cone}
Recall from ~\ref{Newton-point-section} that for any $\ga\in G(F)^{\mathrm{rs}}$, we defined its Newton point 
\[\nu_\ga\in \La_{0,\bQ}^+\times X_*(G_{\mathrm{ab}})\subset\La_\bQ^+\]
where $\La_{0,\bQ}^+$ is the dominant coweight cone in $\La_{0,\bQ}:=\La_0\otimes_\bZ\bQ$ and $\La_\bQ^+=\La_{0,\bQ}^+\times X_*(G_{\mathrm{ab}})_\bQ$.\par 

Let $\sfD\subset\La_{0,\bQ}$ be the positive coroot cone. In other words, $\sfD$ consists of $\bQ$-linear combinations of simple coroots with non-negative coefficients. Note that $\La_{0,\bQ}^+\subset\sfD$. \par
For $\la\in\La^+$, we define the \emph{dominant weight polytope} to be:
\[\sfP_\la:=\La_\bQ^+\cap\mathrm{Conv}(W\cdot\la)=\La_\bQ^+\cap (\la-\sfD).\]
where $\mathrm{Conv}(W\cdot\la)$ denotes the convex hull of the $W$-orbit of $\la$. 
\begin{lem}\label{polytope-intersection-lem}
For each $\la_1,\la_2\in\La^+$ with $\det(\varpi^{\la_1})=\det(\varpi^{\la_2})$, there exists $\mu\in\La^+$ such that $\mu\le\la_1,\mu\le\la_2$ and 
\[(\la_1-\sfD)\cap (\la_2-\sfD)=\mu-\sfD.\]
In particular, we have $\sfP_{\la_1}\cap\sfP_{\la_2}=\sfP_\mu$.
\end{lem}
\begin{proof}
By the assumption $\det(\varpi^{\la_1})=\det(\varpi^{\la_2})$,  $\la_1-\la_2$ lies in the coroot lattice since the derived group $G_0$ is simply connected. Then there exists a partition of the set of simple coroots $\Delta^\vee=\Delta_1^\vee\sqcup\Delta_2^\vee$ such that 
\[\la_1-\la_2=\beta_1-\beta_2\]
where $\beta_i$ is a non-negative integral linear combinations of simple coroots in $\Delta_i^\vee$ for $i\in\{1,2\}$. Let $\Delta=\Delta_1\sqcup\Delta_2$ be the corresponding partion of the set of simple roots. Consider the coweight $\mu:=\la_1-\beta_1=\la_2-\beta_2.$ Then clearly $\mu\le\la_1$ and $\mu\le\la_2$.\par 
We claim that $\mu\in \La^+$. Take any simple root $\alpha\in\Delta_1$. Since $\beta_2$ is positive linear combination of coroots in $\Delta_2$, we have $\langle\alpha,\beta_2\rangle\le0$ and hence $\langle\mu,\alpha\rangle=\langle\la_2-\beta_2,\alpha\rangle\ge0$. Similarly, using $\mu=\la_1-\beta_1$, we see that for all $\alpha\in\Delta_2$, $\langle\mu,\alpha\rangle\ge0$. Thus we conclude that $\mu\in \La_+$.\par 
It is clear that $\mu-\sfD\subset(\la_1-\sfD)\cap(\la_2-\sfD)$. Now we prove the reverse inclusion. Let $\nu\in(\la_1-\sfD)\cap(\la_2-\sfD)$. Then for $i\in\{1,2\}$, $\la_i-\nu\in\sfD$ is a non-negative $\bQ$-linear combination of simple coroots and we need to show that $\mu-\nu\in\sfD$. For any fundamental weight $\omega$, there exists $i\in\{1,2\}$ so that $\omega$ is orthogonal to all coroots in $\Delta_i^\vee$. Without loss of generality assume $i=1$, then we have
\[\langle\mu-\nu,\omega\rangle=\langle\la_1-\beta_1-\nu,\omega\rangle=\langle\la_1-\nu,\omega\rangle\ge 0.\]
This means that $\nu\le_\bQ\mu$, or $\nu\in(\mu-\sfD)$. Therefore we have shown that $\mu-\sfD=(\la_1-\sfD)\cap(\la_2-\sfD)$.\par 
Finally, taking intersection with $\La_\bQ^+$, we get $\sfP_{\la_1}\cap\sfP_{\la_2}=\sfP_{\mu}$.
\end{proof}

For each $\la\in\La^+$, define
\begin{equation}\label{polytope-strata-eq}
\sfP_\la^\circ:=\sfP_\la-\bigcup_{\substack{\mu\in \La^+,\\
\mu<\la}} \sfP_\mu.
\end{equation}

\begin{cor}\label{polytope-strata-cor}
For any $\la_1,\la_2\in\La_+$ with $\la_1\ne\la_2$, we have
$\sfP_{\la_1}^\circ\cap\sfP_{\la_2}^\circ=\varnothing$. In particular, we get a well-defined stratification
\[\La_{0,\bQ}^+\times X_*(G_\mathrm{ab})=\bigsqcup_{\la\in\La^+}\sfP_\la^\circ.\]
\end{cor}
\begin{proof}
If $\det(\varpi^{\la_1})\ne\det(\varpi^{\la_2})$, it is clear that $\sfP_{\la_1}$ and $\sfP_{\la_2}$ are disjoint. Suppose $\det(\varpi^{\la_1})\ne\det(\varpi^{\la_2})$. Then  by Lemma~\ref{polytope-intersection-lem}, there exists $\mu\in\La^+$ such that $\mu\le\la_1,\mu\le\la_2$ and
\[\sfP_{\la_1}^\circ\cap\sfP_{\la_2}^\circ\subset\sfP_{\la_1}\cap\sfP_{\la_2}=\sfP_\mu.\]
But by \eqref{polytope-strata-eq}, we have $\sfP_\mu\cap\sfP_{\la_i}^\circ=\varnothing$ since $\mu\le\la_i$ for $i\in\{1,2\}$. Therefore $\sfP_{\la_1}^\circ\cap\sfP_{\la_2}^\circ=\varnothing$.
\end{proof}

\subsubsection{Stratification on Steinberg base}
To get a conceptually simpler formulation of the conjecture on irreducible components, we introduce a stratification on $\fC(F)$.\par 
Recall from ~\ref{C-le-la-section} that we defined 
$\fC_{\le\la}=\chi(L^+G\varpi^\la L^+G)\subset\fC(F)$. We know that
\[\fC_{\le\la}=\fC_+^\la(\cO)\times\det(\varpi^\la)G_{\mathrm{ab}}(\cO).\]
where
\begin{equation}\label{C-plus-la-O-eq}
\fC_+^\la(\cO)=\bigoplus_{i=1}^r\varpi^{\langle\omega_i,w_0(\lambda)\rangle}\cO \subset F^r\cong\fC_{G_0}(F).
\end{equation}
In particular for each dominant coweight $\mu\in\Lambda^+$ with $\mu\le\lambda$, we get $\langle\omega_i, w_0(\mu)\rangle\ge\langle\omega_i,w_0(\lambda)\rangle$ and $\det(\varpi^\mu)=\det(\varpi^\la)$. Hence there is a natural inclusion $\fC_{\le\mu}\subset\fC_{\le\lambda}$.
\begin{prop}\label{Steinberg-strata-intersection-prop}
For any $\la,\la'\in\La^+$ with $\det(\varpi^{\la})=\det(\varpi^{\la'})$, there exists $\mu\in\La^+$ such that $\mu\le\la,\mu\le\la'$ and 
\[\fC_{\le\la}\cap\fC_{\le\la'}=\fC_{\le\mu}.\]
\end{prop}
\begin{proof}
By Lemma~\ref{polytope-intersection-lem}, there exists $\mu\in\La^+$ such that 
\begin{equation}\label{widecone-intersection-eq}
(\la-\sfD)\cap (\la'-\sfD)=\mu-\sfD
\end{equation}
To prove the proposition, it suffices to show that 
\[\fC_+^{\la}(\cO)\cap\fC_+^{\la'}(\cO)\subset\fC_+^\mu(\cO)\]
Let $\iota$ be the involution on the set $\{1,\dotsc,r\}$ such that $\omega_{\iota(i)}=-w_0(\omega_i)$ for all $1\le i\le r$. For each $c=(c_1,\dotsc,c_r)\in\fC_{G_0}(F)\cong F^r$,
let $a_i:=-\mathrm{val}(c_{\iota(i)})$. Suppose that $c\in\fC_+^{\la}(\cO)\cap\fC_+^{\la'}(\cO)$, then by \eqref{C-plus-la-O-eq} we get
\begin{equation}\label{a-i-inequality-eq}
a_i\le\langle\la,\omega_i\rangle\text{  and  } a_i\le\langle\la',\omega_i\rangle\quad\text{for all }1\le i\le r
\end{equation}
and we need to show that $a_i\le\langle\mu,\omega_i\rangle$ for all $1\le i\le r$.\par 
Let $\la_0:=\sum_{i=1}^r\langle\la,\omega_i\rangle\alpha_i^\vee$ and define $\la_0'$ (resp. $\mu_0$) in a similar way, replacing $\la$ by $\la'$ (resp. $\mu$). Then in particular $\la_0,\la_0',\mu_0\in\La_0^+$. Consider the coweight $\nu:=\sum_{i=1}^r a_i\alpha_i^\vee\in\La_0$. By \eqref{widecone-intersection-eq} and \eqref{a-i-inequality-eq} we have
\[\nu\in(\la_0-\sfD)\cap(\la_0'-\sfD)=\mu_0-\sfD.\]
This implies that
\[a_i=\langle\nu,\omega_i\rangle\le\langle\mu_0,\omega_i\rangle=\langle\mu,\omega_i\rangle\]
and we are done.
\end{proof}
Define
\begin{equation}\label{Steinberg-strata-eq}
\fC_\la^\circ:=\fC_{\le\la}-\bigcup_{\substack{\mu\in\Lambda_+\\ \mu<\la}}\fC_{\le\mu}.
\end{equation}

\begin{cor}
For any $\la_1,\la_2\in\La^+$ with $\la_1\ne\la_2$, we have
$\fC_{\la_1}^\circ\cap\fC_{\la_2}^\circ=\varnothing$. In particular, we get a well-defined stratification
\[\fC(F)=\bigsqcup_{\la\in\Lambda_+}{\fC}^\circ_\la\]
\end{cor}
\begin{proof}
The argument is similar to the proof of Corollary~\ref{polytope-strata-cor}, using Proposition~\ref{Steinberg-strata-intersection-prop} instead of Lemma~\ref{polytope-intersection-lem}.
\end{proof}
The following lemma relates the stratification \eqref{Steinberg-strata-eq} on $\fC(F)$ and the stratification \eqref{polytope-strata-eq} on the dominant weight cone $\La_{0,\bQ}^+\times X_*(G_{\mathrm{ab}})$.

\begin{lem}\label{integral-approximation-lem}
For any $\ga\in G(F)^\mathrm{rs}$, there exists a \emph{unique} dominant integral coweight $\mu\in\La_+$ that satisfies any (or all) of the following equivalent conditions:
\begin{enumerate}
\item $\mu\in\La_+$ is a smallest dominant integral coweight such that $\nu_\ga\le_\bQ\mu$;
\item $\nu_\ga\in\sfP_\mu^\circ$, cf. \eqref{polytope-strata-eq};
\item $\chi(\ga)\in\fC_\mu^\circ$, cf. \eqref{Steinberg-strata-eq}.
\end{enumerate}
\end{lem}
\begin{proof}
The equivalence between (1) and (2) folows from the definition of $\sfP_\mu$. The equivalence of (1) and (3) follows from Proposition~\ref{nonempty-prop}.\par 
Finally, the uniqueness of $\mu$ follows from Lemma~\ref{polytope-intersection-lem} or Proposition~\ref{Steinberg-strata-intersection-prop}.
\end{proof}
Now we state our conjecture on irreducible components of $X_\ga^\la$:
\begin{conjecture}\label{irr-components-conjecture}
Let $\ga\in G(F)^{\mathrm{rs}}$ and $\mu\in\La_+$ be the unique dominant coweight that satisfies the equivalent conditions in Lemma~\ref{integral-approximation-lem}. Then the number of $G_\ga(F)$-orbits on  $\mathrm{Irr}(X_\ga^\la)$ equals to the weight multiplicity $m_{\lambda\mu}$.
\end{conjecture}
By Corollary~\ref{dim-unr-cor}, this conjecture is true when $\ga$ is an unramified conjugacy class. Also, it is true when $\la=0$ by \citep[\S4.2]{Bou15}.
\begin{rem}\label{best-integral-approx-rem}
For irreducible components of affine Deligne-Lusztig varieties, there is a similar conjecture made by Chen-Zhu, see the discussion in \citep{HaVi17} and \citep{XiaoZhu17}. In their setting, they also approximate Newton points of twisted conjugacy classes by integral coweight. However, the ``best integral approximation" as defined in \citep{HaVi17} is the largest integral coweight dominated by the Newton point. Whereas in the formulation of Conjecture~\ref{irr-components-conjecture}, we use the smallest integral coweight dominating the Newton point. Simple examples suggest that these two integral approximations are very likely in the same Weyl group orbit, so we expect the two weight multiplicities to be the same.
\end{rem}

\subsection{Components of the regular locus}\label{regular-components-section}
The $G_\ga(F)$-orbits on $\mathrm{Irr}(X_\ga^{\la,\mathrm{reg}})$ corresponds bijectively to $G_\ga(F)$ orbits on $X_\ga^{\la,\mathrm{reg}}$, which are precisely the $\cP_a$-orbits of maximal dimension on $X_\ga^\la$. We know from Proposition~\ref{X-ga-w-torsor-prop} that these are the varieties $X_\ga^{\la,w}$ for $w\in\mathrm{Cox(W,S)}$.\par  
However, for two different $w,w'\in\mathrm{Cox}(W,S)$, $X_\ga^{\la,w}$ and $X_\ga^{\la,w'}$ might coincide. For example, in the case $\la=0$ and $\ga\in G(\cO)$, all $X_\ga^{\la,w}$ coincide (hence equal to $X_\ga^{\la,\mathrm{reg}}$). So in this particular case $X_\ga^{\la,\mathrm{reg}}$ is the unique $\cP_a$-orbit of maximal dimension. In general, we know from \eqref{X-ga-reg-open-cover-eq} that the number of $G_\ga(F)$ orbits in $X_\ga^{\la,\mathrm{reg}}$ is bounded above by the Cardinality of $\mathrm{Cox}(W,S)$. We will see that in many situations, this upper bound can be achieved (in other words $X_\ga^{\la,w}$ are mutually disjoint).  
\begin{thm}\label{regular-component-upper-bound-thm}
Let $\ga\in G(F)^{\mathrm{rs}}$ be a regular semisimple element. Let $\mu\in\La^+$ be the unique smallest integral coweight that dominantes the Newton point $\nu_\ga$ as in Lemma~\ref{integral-approximation-lem}. Then for any $\la\in\La^+$ with $\la\ge\mu$, $X_\ga^\la$ is nonempty and we have an inequality
\[|\{G_\ga(F) \text{ orbits on }X_\ga^{\la,\mathrm{reg}}\}|\le|\mathrm{Cox}(W,S)|\]
where $\mathrm{Cox}(W,S)$ is the set of $S$-Coxeter elements defined in Definition~\ref{coxeter-definition}. Moreover, when $\la$ lies in the interior of the Weyl chamber and $\la-\mu$ lies in the interior of the positive coroot cone, the equality is achieved. 
\end{thm}
\begin{proof}
It remains to show the last statement. Suppose $\la$ lies in the interior of the Weyl chamber and $\la-\mu$ lies in the interior of the dominant coroot cone. Without loss of generality, we may assume that $\det(\ga)=\det(\varpi^\la)$ and hence obtain an element $\ga_\la\in V_G^\la(F)$ as in Lemma~\ref{gamma-lambda-lem}. Let $a=\chi_\la(\ga_\la)\in\fC_+^\la(\cO)$ and consider the following Cartesian diagram
\[\xymatrix{
\chi^{-1}_\la(a)\ar[r]\ar[d] & V_G^\la\ar[d]^{\chi_\la}\\
\spec\cO\ar[r]^a & \fC_+^\la
}\]
For $g\in G(F)$ with $gK\in X_\ga^{\la,\mathrm{reg}}$ where $K=G(\cO)$, let $\overline{g^{-1}\ga g}$ be the reduction mod $\varpi$ of $g^{-1}\ga g\in V_G^{\la,\mathrm{reg}}(\cO)$. The condition that $\la$ lies in the interior of the Weyl chamber means that $\langle\la,\alpha_i\rangle>0$ for all simple roots $\alpha_i$. Hence the special fiber of $V_G^\la$ is the asymptotic semigroup $\mathrm{As}(G)$ and $\overline{g^{-1}\ga g}\in\mathrm{As}(G)$.\par
The assumption that $\la-\mu$ lies in the interior of the positive coroot cone implies that $\langle\la-\mu,\omega_i\rangle>0$ for all fundamental weight $\omega_i$. Therefore the reduction mod $\varpi$ of $a$ equals to $0$ and the special fiber of $\chi_\la^{-1}(a)$ is the nilpotent cone $\cN$. In particular, we get $\overline{g^{-1}\ga g}\in\cN^{\mathrm{reg}}$.\par 
Consequently there is a bijection between $G_\ga(F)$ orbits on $X_\ga^{\la,\mathrm{reg}}$ and $G$ orbits on $\cN^{\mathrm{reg}}$, the latter of which corresponds bijectively to $\mathrm{Cox}(W,S)$ by Proposition~\ref{nilp-cone-prop}.
\end{proof}

As an immediate consequence, we mention the following purely combinatorial result, which might be of independant interest:
\begin{cor}\label{lower-bound-multiplicity-cor}
Let $\la\ge\mu$ be dominant weights of $\hat G$. Suppose that $\la$ lies in the interior of the Weyl chamber and $\la-\mu$ lies in the interior of the positive root cone (the ``wide cone"). Then we have the following lower bound for the weight multiplicity
\[m_{\la\mu}\ge|\mathrm{Cox}(W,S)|\]
where the set $\mathrm{Cox}(W,S)$ is defined in \S\ref{coxeter-definition}.
\end{cor}
\begin{proof}
We take an unramified element $\ga\in\varpi^\mu T(\cO)\cap G(F)^{\mathrm{rs}}$. Then by Corollary~\ref{dim-unr-cor}, the number of $G_\ga(F)$-orbits on $\mathrm{Irr}(X_\ga^\la)$ equals to $m_{\la\mu}$. On the other hand, by Theorem~\ref{regular-component-upper-bound-thm}, the number of $G_\ga(F)$-orbits on $\mathrm{Irr}(X_\ga^{\la,\mathrm{reg}})$ equals to $|\mathrm{Cox}(W,S)|$, hence the inequality.
\end{proof}
\begin{rem}
If the derived group $G_0$ is simple of rank $r$, then $|\mathrm{Cox}(W,S)|=2^{r-1}$. In general, if the simple factors of $G_0$ has rank $r_1,\dotsc,r_m$, then 
\[|\mathrm{Cox}(W,S)|=\prod_{i=1}^m2^{r_i-1}.\]
We expect that there should be a more straightforward proof of Corollary~\ref{lower-bound-multiplicity-cor}. 
\end{rem}
\begin{rem}\label{regular-components-rem}
In general, the weight multiplicity $m_{\la\mu}$ will increase with $\la$ while the right hand side in Corollary~\ref{lower-bound-multiplicity-cor} is a fixed constant independant of $\la,\mu$. Thus in general there will be much more irreducible components in $X_\ga^\la$ than the regular open subvariety $X_\ga^{\la,\mathrm{reg}}$.
\end{rem}

\bibliographystyle{alpha}
\bibliography{myref}
\end{document}